\documentclass[12pt,reqno]{amsart}

\newcommand\version{December 17, 2024}

\usepackage{amsmath,amsfonts,amsthm,amssymb,amsxtra}
\usepackage{bbm} 
\usepackage{hyperref}	

\usepackage{xfrac}
\usepackage{todonotes, relsize, pgfplots, tikz, float}
\usetikzlibrary{calc,decorations.markings}

\newtheorem{theorem}{Theorem}
\newtheorem{proposition}[theorem]{Proposition}
\newtheorem{lemma}[theorem]{Lemma}
\newtheorem{corollary}[theorem]{Corollary}

\theoremstyle{definition}

\theoremstyle{remark}

\numberwithin{equation}{section}

\renewcommand{\epsilon}{\varepsilon}

\newcommand{\loc}{{\rm loc}}

\newcommand{\N}{\mathbb{N}}

\renewcommand{\phi}{\varphi}
\newcommand{\R}{\mathbb{R}}

\newcommand{\lo}{\mathrm{lo}}
\newcommand{\me}{\mathrm{med}}
\newcommand{\hi}{\mathrm{hi}}

\DeclareMathOperator{\dist}{dist}
\DeclareMathOperator{\Div}{div}

\DeclareMathOperator{\tr}{Tr}
\DeclareMathOperator{\Ric}{Ric}
\DeclareMathOperator{\vol}{vol}

\usepackage{mathtools, todonotes, thmtools}\usepackage{tikz}
\usetikzlibrary{datavisualization}
\usetikzlibrary{datavisualization.formats.functions}
\usetikzlibrary{patterns}

\let\oldtocsection=\tocsection

\let\oldtocsubsection=\tocsubsection

\let\oldtocsubsubsection=\tocsubsubsection

\renewcommand{\tocsection}[2]{\hspace{0em}\oldtocsection{#1}{#2}}
\renewcommand{\tocsubsection}[2]{\hspace{1em}\oldtocsubsection{#1}{#2}}
\renewcommand{\tocsubsubsection}[2]{\hspace{2em}\oldtocsubsubsection{#1}{#2}}

\begin{document}
	
	\begin{titlepage}
		\huge \title[Sharp $\sigma_k$-curvature inequality]{The sharp $\sigma_k$-curvature inequality on locally conformally flat manifolds in quantitative form}
		\vspace{7cm}
	\end{titlepage}

	\author{Jonas W.~Peteranderl}
	
	\address[Jonas W.~ Peteranderl]{Mathematisches Institut, Ludwig-Maximilians-Universit\"at M\"unchen, Theresienstr.~39, 80333 M\"unchen, Germany}	\email{peterand@math.lmu.de}

\date{\version}
	
	\begin{abstract}
    Let $2\leq k<n/2$ and let $(M^n,[g])$ be a smooth, closed, connected, and locally conformally flat Riemannian manifold with a $k$-admissible metric in the conformal class $[g]$. We prove a stability result of the $\sigma_k$-curvature inequality on $M$, in the sense that if equality is almost satisfied for some conformal metric, then its conformal factor is close to a minimizer of the inequality. Closeness is measured quantitatively in terms of Sobolev norms of the conformal factor, namely with respect to the $W^{1,2}$- and the $W^{1,2k}$-norm. In the non-degenerate case, these norms come with optimal exponents $2$ and $2k$, respectively, whereas in general the exponents are $2+\gamma$ and $\max\{2k,2+\gamma\}$ for some $\gamma\geq 0$ originating from a {\L}ojasiewicz inequality. This extends a previous result by Frank and the author from $k=2$ and the sphere to $2\leq k<n/2$ and the full class of manifolds originally considered by Viaclovsky. It also extends a previous result by Engelstein--Neumayer--Spolaor from $k=1$ to the setting of fully non-linear scalar curvatures.
	\end{abstract}
	
	\date{\today}
	\thanks{\copyright\, 2026 by the author. This paper may be reproduced, in its entirety, for non-commercial purposes.}
	
	\maketitle
	\setcounter{page}{1}

	\section{Introduction and main result}
	
	We consider smooth Riemannian manifolds $(M^n,g)$, $n\geq 3$. 
    The associated \textit{Schouten tensor}, denoted by $$A_g\coloneqq \frac{1}{n-2} \left(\Ric_g-\frac{R_g}{2(n-1)} g\right)$$ emerges in the Weyl decomposition of the Riemannian curvature tensor. If $A$ is an $n\times n$-matrix, then $\sigma_k(A)$ denotes the $k$-th coefficient of the characteristic polynomial of $(-A)$ as given by Vieta's formula $$\det(\mathbbm 1+tA)=\sum_{k=0}^n \sigma_k(A)t^k\,.$$ If $A$ is diagonalizable, $\sigma_k(A)$ is the $k$-th elementary symmetric polynomial of the eigenvalues of $A$. The quantities $\sigma_k(A)$ are well-known invariants of the matrix, most notably $\sigma_1(A)=\tr A$ and $\sigma_n(A)=\det A$.
    
    As a natural generalization of the scalar curvature $R_g=2(n-1)\sigma_1(g^{-1}A_g)$ in the context of the Yamabe problem, Viaclovsky \cite{Via00} introduced the notion of \textit{$\sigma_k$-(scalar) curvatures} \begin{equation}
        \label{eq:sigma_k_def}\sigma_k(g)\coloneqq \sigma_k(g^{-1}A_g)\coloneqq  \sigma_k(\lambda(g^{-1}A_g))\,,
    \end{equation} that is, the $k$-th elementary symmetric polynomial of the eigenvalues $\lambda=\lambda(g^{-1}A_g)$ of the Schouten tensor. For $k=0$ we set $\sigma_0=1$. The abuse in notation in \eqref{eq:sigma_k_def} should cause no confusion as the entries lie in different spaces.  A metric is called \textit{$k$-admissible}, $k\in \N$, if the eigenvalues of $g^{-1}A_g$ lie in the \textit{G\aa rding cone} $$\Gamma_{k}^+\coloneqq \{ \lambda\in \R^n : \sigma_j(\lambda)>0, 1\leq j\leq k\}\,.$$ For any metric $h$ on $M$, we write $$\mathcal F_k[h]\coloneqq \frac{1}{\vol_h(M)^{\frac{n-2k}{n}}}\int_M\sigma_k(h)\,\mathrm d\vol_h$$ for the total normalized $\sigma_k$-curvature. If $k\neq n/2$ and either $M$ is locally conformally flat or $k\leq 2$, finding critical points of this functional corresponds by \cite{Via00} to the Yamabe problem for $\sigma_k$-curvatures: Find a metric $\tilde g$ conformal to $g$ such that $\sigma_k(\tilde g)$ is constant. For $k=1$ this corresponds to the original Yamabe problem, which was solved in consecutive works \cite{Yam60,Tru68,Aubin1976,Sch84}. On locally conformally flat manifolds, existence was proved independently by Guan--Wang \cite{GW04} using a conformal flow and Li--Li \cite{Li2005} using elliptic methods. If the problem is variational, then existence was proved for $k\leq n/2$ in  \cite{ShengTrudingerWang2007}; though by \cite{BransonGover2008}, the problem is variational exactly if the manifold is locally conformally flat or $k\leq 2$. For $k>n/2$, existence was proved in \cite{GV07}, and for $k=n/2$ with $n=4$ in \cite{Chang2002, Chang2002a} and for $n$ even in \cite{TrudingerWang2010}. In contrast the problem is still open in case of not locally conformally flat manifolds and $3\leq k<n/2$.
    \vspace{0.2cm}
     
      Every metric in the conformal class $[g]$ of $g$ is given by $$g_u\coloneqq u^{\frac{4k}{n-2k}}g$$ for some smooth positive function $u:M\to \R$. 
     Then we have
     $\mathrm d {\vol}_{g_u}  =u^{q}\mathrm d {\vol}_{g}$, where  we use in the following  the convention that $$q\coloneqq \frac{2kn}{n-2k}\,.$$ Introducing the notation $$F_{k,g}[u]\coloneqq \mathcal F_k[g_u]\qquad \text{and}\qquad\mathcal C_k([g])\coloneqq \{u\in C^\infty(M;\R) : u>0 \ \text{with} \ g_u \ \text{$k$-admissible}\}\,,$$ the $\sigma_k$-Yamabe constant is
     $$ Y_k= Y_k(M,[g])\coloneqq \inf_{u\in \mathcal C_k([g])} F_{k,g}[u]\,.$$ Then we call 
     $$F_{k,g}[u]\geq Y_k$$ the \textit{$\sigma_k$-curvature inequality} for $u\in \mathcal C_k([g])$. We denote by $L^r(g)$ the Lebesgue space on $(M,g)$ with exponent $r$ and  by $W^{s,r}(g)$ the Sobolev space on $(M,g)$ with exponents $s$ and $r$. The set of normalized minimizers $$\mathcal M_k\coloneqq \{v\in \mathcal C_k([g]):\|v\|_{L^q(g)}=1\,,\quad F_{k,g}[v]=Y_k(M,[g])\}$$ is non-empty; see Proposition \ref{prop:cpct}. Note that it does not have to be a manifold in general.

    We denote the group of conformal diffeomorphisms $\Psi$ of $(M,[g])$ by $G$ and define the Jacobian $J_\Psi=|\det D\Psi|$ via $\Psi^*\, \mathrm d \vol_{g}= J_\Psi \, \mathrm d \vol_{g}$. The transformation \begin{equation}
        \label{eq:conf}(u)_\Psi\coloneqq J_\Psi^{\frac 1q}(u\circ \Psi)\,,\qquad \Psi \in G\,,\end{equation} leaves the volume and the total $\sigma_k$-curvature, and thus $F_{k,g}$, invariant, that is $F_{k,g}[(u)_\Psi]=F_{k,g}[u]$. Note that $((u)_{\Psi_1})_{\Psi_2}=(u)_{\Psi_1\circ\Psi_2}$.

The case $k=1$ recovers (up to a dimensional constant) the classical Yamabe functional. In particular, on the round sphere $(\mathbb S^n,g_*)$ the corresponding variational inequality  is equivalent to the sharp Sobolev inequality on Euclidean space. The first genuinely fully nonlinear case is $k=2$. Frank and the author \cite{FrankPeteranderl2024_sigma2} proved on the sphere that if the $\sigma_2$-curvature inequality is almost an equality -- or equivalently if the volume-normalized total $\sigma_2$-curvature is almost minimal -- for some admissible metric,  then this metric is almost a minimizer. Their result gives a quantitative refinement of the $\sigma_2$-curvature inequality: There is a $c_n>0$ such that
    \begin{equation}
       F_{2,g_*}[u]-Y_2(\mathbb S^n,[g_*]) \geq c_n\inf_{\lambda>0,\Psi\in G} (\|\lambda(u)_\Psi-1\|^{2}_{W^{1,2}(g_*)}+\|\lambda (u)_\Psi-1\|^{4}_{W^{1,4}(g_*)})\label{eq:sigma_2}\end{equation} for every smooth $u>0$ with $\sigma_1(g_u)>0$.

    Another aspect that enters for general locally conformally flat manifolds is that the kernel of the conformal Hessian $Q_{g}$ of $F_{k,g}$ does not have to coincide with the tangent space of the set of minimizers; cf.~Subsection \ref{subsec:conf_Hess}. On closed Riemannian manifolds not conformally equivalent to the sphere, Engelstein--Neumayer--Spolaor \cite{EnNeSp} proved (with an equivalent notion of distance) that there are constants $c>0$ and $\gamma\geq 0$ depending on $(M,[g])$ such that \begin{equation}\label{eq:Yamabe_1}F_{1,g}[u]-Y_1(M,[g])\geq c \inf_{\lambda>0,v\in \mathcal M_1} \|\lambda u-v\|_{W^{1,2}}^{2+\gamma}\,.\end{equation}
    
    To guarantee $\gamma=0$, we use the following notion. We say that the set of normalized minimizers $\mathcal M_k$ is \textit{non-degenerate} if the set $\R_+\mathcal M_k$ is locally a finite-dimensional smooth manifold and it holds that $$\ker Q_{g_v}=  v^{-1}T_v(\R_+\mathcal M_k)$$ 
    for all $v\in \mathcal M_k$.  It turns out to be convenient to work with functions relative to the minimizer, which explains the factor $v^{-1}$ in the above definition of non-degeneracy. After removing the multiplication by constants and the relative factor $v^{-1}$, note that this notion of non-degeneracy is equivalent to the notion of integrability in \cite{AS88,CCR15,EnNeSp} in case the Hessian is self-adjoint and elliptic and the functional is analytic; see \cite[Subsection 1.5]{Fee20}.
 
In the following, we will restrict to $2\leq k< n/2$. We can now state our main theorem.

\begin{theorem}[Quantitative stability]\label{thm:1} Let $2\leq k< n/2$ and $(M^n,[g])$ be a smooth, closed, connected, locally conformally flat Riemannian manifold with $k$-admissible metric $g$. Then there are $\gamma=\gamma(M,[g],k)\geq 0$ and $c=c(M,[g],k)>0$ such that
    \begin{equation*}
       F_{k,g}[u]-Y_k(M,[g]) \geq c\inf_{\lambda>0,\Psi\in G\,,  v\in \mathcal M_k} (\|\lambda v^{-1}(u)_\Psi-1\|^{2+\gamma}_{W^{1,2}(g_v)}+\|\lambda v^{-1}(u)_\Psi-1\|^{\max\{2k,2+\gamma\}}_{W^{1,2k}(g_v)})
    \end{equation*} for every $u\in \mathcal C_k([g])$. If $\mathcal M_k$ is non-degenerate, then we may take $\gamma=0$. 
\end{theorem}

This theorem generalizes the case $(M^n,[g])=(\mathbb S^n,g_*)$ and $k=2$ from \cite{FrankPeteranderl2024_sigma2}. We also mention a recent result for $k=2$ and Einstein metrics by Wu \cite{Wu26}. Also our notion of distance is a natural extension of that in \cite{FrankPeteranderl2024_sigma2}. Indeed, in the non-degenerate case, the two-term distance with quadratic $W^{1,2}$-norm and quartic $W^{1,4}$-norm is replaced by a distance with quadratic $W^{1,2}$-norm and the $2k$-th power of the $W^{1,2k}$-norm. The latter non-quadratic exponent is related to the Figalli--Zhang phenomenon \cite{Figalli2022}. Moreover, the distance here includes a factor $v^{-1}$, which accounts for the fact that the minimizers are not necessarily constants up to M\"obius transformations in a general conformal class, but could potentially be non-constant and non-unique even after conformal diffeomorphisms.

The exponents $2$ and $2k$ are optimal on the sphere. In general, the possible loss in the exponents in terms of $\gamma\geq 0$ is related to the very same phenomenon that Engelstein--Neumayer--Spolaor \cite{EnNeSp} observed in the setting of the Yamabe inequality, which is the case $k=1$ here. Their local argument combines a Lyapunov--Schmidt reduction with a {\L}ojasiewicz inequality. The same finite-dimensional mechanism applies here since $F_{k,g}$ is real-analytic in a $C^{2,\alpha}$-neighborhood of positive functions. A novel phenomenon in Theorem \ref{thm:1} is that this mechanism interacts with the nonlinear Figalli--Zhang phenomenon in terms of the $W^{1,2k}$-norm through the maximum $\max\{2k,2+\gamma\}$. An optimal {\L}ojasiewicz exponent $2+\gamma$ causes a lower bound on the sharp exponents in our stability result, which will be further discussed in Section~\ref{sec:5}. In Lemma \ref{lem:dist_notions} we will see that the distances in \eqref{eq:sigma_2} and \eqref{eq:Yamabe_1} are special cases of our notion of distance in Theorem \ref{thm:1}.

So far only compactness of almost minimizing critical points was known by Li and Li \cite{Li2005}. Our stability result improves this to compactness of almost minimizing admissible functions and gives a quantitative rate of convergence in terms of the deficit.

An interesting future direction, which is beyond the scope of this article, is to identify a conformal class for which a positive $\gamma$ is necessary. A promising candidate is given in Subsection~\ref{subsec:comp}, for which the required classification of $\sigma_k$-minimizers remains to be established. Although plausible, we do not claim here that the case $\gamma=0$ is generic in the sense of \cite{EnNeSp}. Their conclusion is based on a result by Schoen \cite{Sch91}, which is currently not available. Another interesting question concerns \cite{GW04}: The volume-normalized total $\sigma_k$-curvatures satisfy the following monotonicity formula
\begin{equation}
    \label{eq:monot}
(\mathcal F_{k}[g])^{\frac 1{n-2k}}\geq C([g],n,k,l) (\mathcal F_{l}[g])^{\frac 1{n-2l}}
\end{equation} for $0\leq l< k<n/2$, which are refinements of the $\sigma_k$-curvature inequalities we were looking at for each $l\geq 1$, so it remains open which stability behavior these $\sigma_k$-$\sigma_l$-curvature inequalities exhibit compared to the case $l=0$. A more structural open question is whether it is possible to extend Theorem \ref{thm:1} to metrics in $\Gamma_{k-1}^+$ as it was done in \cite{FrankPeteranderl2024_sigma2} for $k=2$ and the sphere; see also \cite{GWW26} for recent progress. Finally, it would be interesting to investigate the stability of reverse $\sigma_k$-curvature inequalities for $n/2<k\leq n$ as initiated in \cite{KP25}.

\subsection{Background from stability of functional inequalities}

Qualitative stability means compactness up to symmetries of optimizing sequences for a given functional inequality, whereas (quantitative) stability gives a control on the distance to the set of optimizers in terms of the deficit.  This question was raised for the Sobolev inequality by Brezis--Lieb \cite{BrLi} and answered affirmatively by Bianchi--Egnell \cite{BiEg} with an optimal power $2$ of the distance; see also \cite{Frank2024} for an introduction. Extensions of this result to the $p$-Sobolev inequality, $p\neq 2$, were developed in \cite{Figalli2019, Neu20} and completed with the optimal exponent $\max\{2,p\}$ by Figalli--Zhang \cite{Figalli2022}; see also \cite{Cianchi2009} for earlier stability results with respect to another distance. For other stability results exhibiting the Figalli--Zhang phenomenon of non-quadratic stability exponents, we refer to \cite{FrankPeteranderl2024_sigma2,KP25, WZ25,FPR2025, FLZ26, FrankLiYang2026}.

The general Bianchi--Egnell strategy is non-constructive: It is based on a two-step approach including a local second order analysis near the Aubin--Talenti minimizers \cite{Rod66, Aubin1976Sobolev,Ta} and a compactness-based global-to-local reduction \`a la Lions \cite{Lio85a,Lio85b}, which together give global stability by a standard contradiction argument. Dolbeault--Esteban--Figalli--Frank--Loss \cite{Dolbeault2025} viewed the global problem through a  monotone deformation via alternating competing symmetries and continuous Steiner symmetrization; see also \cite{Dolbeault_2025_2402} for a review, and \cite{chen_2024,chen_2025,chen_2025b,ChenLuTangWang2026}. The last of those articles provides a variant of this argument using the CR-Yamabe flow when rearrangement is unavailable.

On closed Riemannian manifolds, Engelstein--Neumayer--Spolaor \cite{EnNeSp} proved quantitative stability for the Yamabe inequality with possibly non-quadratic stability exponents in non-generic cases. We refer to \cite{andrade2024} for a recent extension to the Yamabe inequality for higher $Q$-curvatures, to \cite{BCV26,LinYu2026} for versions with manifolds with boundary, and to \cite{AikenBorquez2026} for a fractional setting. Frank \cite{Fr} analyzed this non-generic case for the critical cylinder and proved a quartic stability result; see also \cite{andrade2024} for the corresponding result for the $Q$-curvature functional and \cite{FrankPeteranderl2024_CKN} for similar observations in case of non-constant optimizers of the CKN-inequality on the Felli--Schneider curve.

\subsection{Comparison with the quantitative $\sigma_2$-curvature inequality on the sphere}\label{subsec:comp}

In \cite{FrankPeteranderl2024_sigma2} the above theorem for $k=2$ and $M$ the round sphere $\mathbb S^n$ was established. Following the Bianchi--Egnell scheme, the proof consisted of a global-to-local reduction and a local analysis. Both are based on refined strategies tailored for the specific problem. We discuss why both strategies fail in the context of higher $k$ and general manifolds. In the next subsection we then explain how to repair or replace them.

The first part was proved by a reduction to the compactness of the Yamabe inequality on the sphere via monotonicity \eqref{eq:monot}, which combined a tailored positive scalar integral density for the total $\sigma_2$-curvature over the sphere
$$ 
\left(\frac{4}{n-4} \right)^3 \left(\left(\frac{n-4}4\right)^2 u^{\frac{2n}{n-4}}\sigma_1(g_u)+\frac12 |\nabla u|^2 +\frac{n-2}{2}\left(\frac{n-4}{4}\right)^2 u^2 \right) |\nabla u|^2
+\frac{n(n-1)}{8}u^4$$
 found by Case \cite{Case2021} and the fact that both the volume-normalized total scalar curvature and $\sigma_2$-curvature have the same set of minimizing conformal metrics, namely with conformal factors given by constant functions up to M\"obius transformations. 

However, on general manifolds, the second variations of the Yamabe and the $\sigma_2$-curvature functionals do not need to be the same. For instance, a phase-plane analysis on the cylinder $\mathbb S^1(T)\times \mathbb S^{n-1}(1)$ with radius $T>0$ tells us that constant conformal factors are have nonnegative second variation of the volume-normalized total scalar curvature if and only if $T\leq (n-2)^{-1/2}$, where the critical threshold coincides with the occurrence of additional zero modes of the Hessian; cf.~\cite{Sc, Frank2024}. Since the corresponding threshold for $k=2$ is $(n-4)^{-1/2}$, we observe that in the regime $$\frac 1{\sqrt{n-2}}<T<\frac 1{\sqrt{n-4}}$$ constant functions are stable in the $\sigma_2$-setting but unstable in the Yamabe setting. This makes a reduction to the Yamabe inequality as in \cite{FrankPeteranderl2024_sigma2} for general manifolds unlikely.

The local part relies on an asymptotic decoupling of the spectral gap of the Hessian and the positivity of the scalar curvature, achieved via a local additive decomposition into spherical harmonics of the form 
$$u_j=1+r_j=1+y_j+z_j\,,$$ where $y_j$ are lower and $z_j$ are higher order spherical harmonics.

Already for $k=3$, the natural Gonz\'alez-type integration by parts approach produces a first Newton tensor $T_1$; see \cite{Gonzalez2005, MaWu2024} and \eqref{eq:Newton} for a definition of Newton tensors. This is not a scalar quantity determined by the $\sigma_k$-curvatures $\sigma_k(g_u)$ or $|\nabla u|^2$ anymore.
Indeed, note that for the matrix $A_L\coloneqq \mathrm{diag}(L,1,\dots,1)$, $L>0$, the Newton tensor is
$$T_1(A_L)=\sigma_1(A_L) \mathbbm 1-A_L=\mathrm{diag}(n-1, L+n-2,\dots, L+n-2)\,,$$ so evaluating $T_1(A_L)$ in two different directions can give different values. In Euclidean coordinates, the metric $u^{4k/(n-2k)} g_{Eucl}=\exp(-(Lx_1^2+x_2^2+\dots+x_n^2)) g_{Eucl}$ realizes $A_L$ as its Schouten tensor at the origin and is $k$-admissible in a neighborhood of that point. While a density only in terms of $\sigma_k$-curvatures and $|\nabla u|^2$ by further induction is possible, the previously described directional dependence arises naturally and makes a density with only positive coefficients unlikely.

\subsection{Proof strategy and structure}

In line with the Bianchi--Egnell method, we prove a local bound first, before deriving the global bound. In contrast to \cite{FrankPeteranderl2024_sigma2}, we use a formulation without sequences, which is more in line with the constructive approach in \cite{Dolbeault2025}.

In Section \ref{sec:3} we prove the following local version of Theorem \ref{thm:1}.

\begin{proposition}
    [Local bound]\label{prop:loc} Let $2\leq k< n/2$ and $(M^n,[g])$ be a smooth, closed, connected, locally conformally flat Riemannian manifold with $k$-admissible metric $g$. There are $\gamma\geq 0$ and constants $c_{\loc},\rho_{\loc}>0$ such that every $u\in \mathcal C_k([g])$
    with 
$$\dist_{k,\gamma}(u)\leq \rho_\loc$$
    satisfies $$F_{k,g}[u]-Y_k\geq c_{\loc} \dist_{k,\gamma}(u)\,.$$
\end{proposition}

As needed for the proof of Theorem \ref{thm:1}, the constants $c_\loc$, $\rho_\loc$, and $\gamma$ depend on $M$, $[g]$, and $k$. The quantity $\dist_{k,\gamma}$ is defined in \eqref{eq:dist_def}.

We prove the local bound by combining the analytic reduction with a multiplicative splitting. At a fixed normalized minimizer $v$ relative to the background metric $g_0$, put $g=g_v=v^{2q/n}g_0$. Lyapunov--Schmidt reduction gives an analytic graph $\alpha_g(\xi)$ over the normalized kernel of the conformal Hessian. The reduced functional $\Sigma_g(\xi)$ controls the distance from $\xi$ to the zero set of $\Sigma_g(\xi)-Y_k$, which parametrizes minimizers up to normalization.

An almost minimizer $u\in \mathcal C_k([g_0])$ admits the decomposition $$u=cv(\alpha_g(\xi)+z+y)=cv(\alpha_g(\xi)+z)(1+\eta)\,,$$ where $z$ contains finitely many frequencies of the Hessian at $g$, $y$ contains the high frequencies, and $\eta\coloneqq y/(\alpha_g(\xi)+z)$. Denote $\tilde g\coloneqq (\alpha_g(\xi)+z)^{2q/n}g$. This yields an exact splitting into a low frequency, a medium frequency, a high frequency, and a mixed frequency contribution given by
\begin{align*}
    F_{k,g_0}[u]-Y_k=&\ F_{k,g}[\alpha_g(\xi)]-Y_k+(F_{k,g}[\alpha_g(\xi)+z]-F_{k,g}[\alpha_g(\xi)]) \notag\\&+ \left(F_{k,\tilde g}[1+\eta]-F_{k,\tilde g}[1]-D_1F_{k,\tilde g}[\eta]\right)+D_{\alpha_g(\xi)+z}F_{k, g}[y]\,.\end{align*}  The Hessian gives a quadratic control on $z$ and the low frequency part. A superadditivity argument for the Newton tensor $T_{k-1}$ along an affine path of a reparametrized conformal factor controls the high frequency part by $\|y\|_{W^{1,2}(g)}^2+\|y\|_{W^{1,2k}(g)}^{2k}$. The {\L}ojasiewicz inequality then gives the powers $2+\gamma$ and $\max\{2k,2+\gamma\}$. Fixed spectral decompositions in finitely many neighborhoods give uniform constants -- even if the kernel dimension varies between minimizers.

    In terms of this decomposition, after removing the zero modes of the Hessian, we develop separate bounds for the low, medium, and high frequencies.

The global bound is then obtained in Section \ref{sec:4}. The proof is based on a non-constructive consequence of the constructive flow-type argument, developed by Dolbeault--Esteban--Figalli--Frank--Loss \cite{Dolbeault2025}, applied to the logarithmic Guan--Wang flow $g_t$ in \cite{GuanWangFlow2003,GW04} starting at an almost-minimizing sequence given by 
$$\frac {d}{dt} g_t=-(\log(\sigma_k(g_t))-\log(r_k(g_t)))g_t\,,$$ where $\log(r_k(g_t))$ is the volume-normalized mean of $\log(\sigma_k(g_t))$. The strategy to consider a parabolic flow to reach closer to a neighborhood of the minimizer was also used in \cite{ChenLuTangWang2026} before. The Guan--Wang flow satisfies several decisive properties such as global existence, conservation of $k$-admissibility and volume, and monotonicity of the total $\sigma_k$-curvature.

Since our proof is based on a compactness result for almost minimizing critical points, a constructive proof with explicit dimensional control on the stability constant is beyond the scope of our present work. As a consequence of our stability result, we can in fact promote this compactness (up to symmetries) of almost minimizing critical points from \cite{Li2005} to compactness (up to symmetries) of general almost minimizing admissible functions. 

\begin{corollary}
    [Compactness of almost minimizers]\label{cor:glob2loc} Under the assumptions of Theorem \ref{thm:1}, let $(u_j)\subset \mathcal C_k([g_0])$ satisfy $$F_{k,g_0}[u_j]\to Y_k(M,[g_0])\qquad \text{and}\qquad \|u_j\|_{L^q(g_0)}=1\,.$$
    Then $$\inf_{\Psi\in G,v\in \mathcal M_k} \|v^{-1}(u_j)_\Psi-1\|_{W^{1,2k}(g_v)}\to 0\,.$$
\end{corollary}

Finally, we prove sharpness of the exponents in Section \ref{sec:5}.

\section{Preliminaries}
\label{sec:2}
In this section we record notions from conformal geometry and prove related  compactness results and global bounds. Here local conformal flatness enters as the Newton tensor is divergence free facilitating integration by parts and through applications of results by Guan--Wang \cite{GW04} and Li--Li \cite{Li2005}. 

In the following, we write $p\in\{2,2k\}$ if statements hold for both norms. We fix a background metric $g_0$ (denoted $g$ in the introduction). The metric $g_0$ is also used as the initial metric of a flow. Unless a change of reference is stated, the set $\mathcal M_k$ consists of conformal factors relative to $g_0$ and $g_v= v^{2q/n}g_0$; in particular, $v\in \mathcal M_k$ always implies $\vol_{g_v}(M)=1$. We use $\|f\|_{W^{1,p}(h)}^p=\int_M (|\nabla_h f|_h^p+|f|^p)\, d\vol_h$ and the convention $\Delta_h=\Div_h\nabla _h$. Applying the Newton tensor, we also specify the metric used. If the index is suppressed, then the background metric is considered.

Positive constants $c,C$ may change from line to line. Unless a more precise dependence is stated, local constants depend on the fixed smooth background metric, $n$, $k$, H\"older exponents, the order of regularity, the chosen finite-dimensional chart, and the spectral cutoff, when used.

\subsection{Parametrizations of the conformal factor}

Similarly to the Yamabe and the $\sigma_2$-parametrizations of the conformal factor in \cite{FrankPeteranderl2024_sigma2}, we use a different parametrization depending on the context. 

Our standard parametrization for the $\sigma_k$-problem is $$g_u=u^{\frac{2q}n} g_0$$ with $q=2kn/(n-2k)$, which turns the total $\sigma_k$-curvature into a $2k$-homogeneous functional in $u$ and the volume into the $q$-th power of the $L^q(g_0)$-norm of $u$, and thus the $\sigma_k$-curvature inequality into a $2k$-Sobolev-type inequality. Since we rely on a consequence of positive scalar curvature in Lemma \ref{lem:Har}, we use the Yamabe parametrization $$u^{\frac {4k}{n-2k}}=W^{\frac 4{n-2}} \,.$$ For most computations involving the Schouten tensor, the inverse parametrization $$u^{\frac {4k}{n-2k}}=V^{-2}$$ is convenient as the term $d u\otimes d u$ disappears in the transformation formula under conformal changes; see \cite{Gonzalez2005, MaWu2024}. Indeed, while for $g_u=u^{2q/n}g$ the transformation is $$g_u^{-1}A_{g_u}=u^{-\frac{2n}{n-2k}}g^{-1}\left(u^2A_g-\frac{2k}{n-2k} u\nabla ^2 u+\frac{2kn}{(n-2k)^2} d u\otimes d u -\frac{2k^2}{(n-2k)^2}|\nabla u|^2 g\right),$$
we write \begin{equation}
    \label{eq:Schouten_V}B_g(V)\coloneqq (V^{-2}g)^{-1}A_{V^{-2}g}= V^2 g^{-1}A_g+V g^{-1}\nabla ^2V-\frac{1}2 |\nabla V|^2\mathbbm 1\,.
\end{equation}
The latter is used for our affine interpolation in the local analysis. These transformation formulas can be obtained from the more standard formulation in logarithmic parametrization; see \cite{LeeParker1987}, for instance.

\subsection{The conformal Hessian}\label{subsec:conf_Hess} When differentiating the functional $F_{k,g_0}$, we have to ensure that the perturbations stay in the conformal class. This leads to the notion of \textit{conformal Hessian}.

To this end, we define the \textit{Newton tensor} 
\begin{equation}\label{eq:Newton}
    T_j(A)\coloneqq \frac{\partial \sigma_{j+1}}{\partial A}= \sum_{l=0}^j (-1)^l \sigma_{j-l}(A)A^l\,,\qquad T_0\coloneqq \mathbbm 1
\end{equation}
for a symmetric matrix $A$. It satisfies $$\tr T_j(A)=(n-j)\sigma_j(A)\qquad \text{and}\qquad \tr(T_{j-1}(A)A)=j\sigma_j(A)$$ and, if $A\in \Gamma_k^+$, then $\sigma_j(A)>0$ and $T_{j-1}(A)$ is positive definite for all $1\leq j\leq k$; see \cite[Lemma 4.2]{Gonzalez2005}, for instance, and also \cite{Garding1959} for the general calculus. Moreover, $T_j(g^{-1}A_g)$ is divergence free with respect to $g$ on a locally conformally flat manifold; see \cite[Lemma 4.3]{Gonzalez2005} and \cite[p.~306, 307]{Via00}.

The first conformal variation tells us that the critical metrics of the volume-normalized total $\sigma_k$-curvature are the metrics with constant $\sigma_k$-curvature so they solve the $\sigma_k$-Yamabe problem; see \cite[Theorem 13]{Via00} for $k\neq n/2$. Following Viaclovsky \cite{Via00}, we derive the second conformal variation of the total $\sigma_k$-curvature in the parametrization $V^{-2}g$, $V>0$, for metrics $g$ that are critical in the sense that their $\sigma_k$-curvature is constant. 

We vary a constant $\sigma_k$-curvature metric by inserting \begin{equation}\label{eq:Vt}
    V_t\coloneqq (1+t\phi)\qquad \text{with}\quad g_t\coloneqq V_t^{-2}g=(1+t\phi)^{-2}g
\end{equation} into the total $\sigma_k$-curvature and differentiating twice. Note that this covers all perturbations of the form $\tilde V_t=V(1+t\phi)$ as we can just absorb $V$ into the background metric $g=g_0$ in computations. Note further that we choose $|t|$ sufficiently small, depending on $\phi\in C^\infty(M)$, so that $V_t>0$.

We differentiate \eqref{eq:Schouten_V} for $V=V_t$. Then $$B_t\coloneqq B_{g}(V_t)=V_t^{2}g^{-1} A_g+tV_tg^{-1}\nabla^2 \phi-\frac {t^2}2|\nabla \phi|_g^2\mathbbm 1\,,$$ where $|\nabla \phi|_g^2\coloneqq g(\nabla \phi,\nabla \phi)$. We drop the $g$ in this notation if there is no confusion. Since $\dot g_t=-2\phi V_t^{-1}g_t\eqqcolon -2 \psi_t g_t$, differentiation gives
\begin{equation}
    \label{eq:BandVol}
\dot B_t =2 \psi_t B_t+g_t^{-1}\nabla_{g_t}^2 \psi_t\qquad \text{and}\qquad \frac{d}{dt} (\mathrm d\vol_{g_t})=-n \psi_t\, \mathrm d \vol_{g_t}\,.\end{equation} Hence, $D\sigma_k(A)[B]=\tr(T_{k-1}(A)B)$ implies
\begin{equation}
    \label{eq:sigmak_der1}
    \frac{d}{dt}\sigma_k(g_t)= 2k\psi_t \sigma_k(g_t)+\tr(T_{k-1}(B_t)g_t^{-1}\nabla_{g_t}^2 \psi_t)\,.
\end{equation}
Denote \begin{equation*}
    \mathcal E(t)\coloneqq \int_M \sigma_k(g_t)\, \mathrm d \vol_{g_t}\,.
\end{equation*} Apply \eqref{eq:BandVol} and \eqref{eq:sigmak_der1} to recover the first variation
\begin{equation}
    \label{eq:E'}
    \mathcal E'(t)= -(n-2k) \int_M \psi_t\sigma_k(g_t)\, \mathrm d \vol_{g_t}\,,
\end{equation} where we integrated by parts using that the Newton tensor is divergence-free. As $\partial_t\psi_t= -\psi_t^2$, differentiating once more gives \begin{equation}
    \label{eq:E''}
    \mathcal E''(t)=(n-2k)\left(\int_M T_{k-1} (B_t)(\nabla_{g_t} \psi_t,\nabla_{g_t} \psi_t)_{g_t} \, \mathrm d\vol_{g_t}+(n-2k+1) \int_M\psi_t^2 \sigma_k(g_t)\,\mathrm d \vol_{g_t}\right)\,.
\end{equation} Both terms are nonnegative whenever $g_t$ is $k$-admissible. Similarly, we compute the derivatives of $\mathcal V (t)\coloneqq \vol_{g_t}(M)$ and find
$$\mathcal V' (t)= -n\int_M\psi_t \, \mathrm d \vol_{g_t}\qquad \text{and}\qquad \mathcal V'' (t)= n(n+1)\int_M\psi_t^2\, \mathrm d \vol_{g_t}\,.$$
Then \begin{align*}
    \frac{ d^2}{dt^2}\mathcal F_k[g_t]=\ & \mathcal V(t)^{-\frac{n-2k}n}\bigg(\mathcal E''(t) -2\frac{n-2k}{n} \mathcal V(t)^{-1}\mathcal V'(t)\mathcal E'(t)\\&+\frac{n-2k}{n}\left(\frac{n-2k}{n}+1\right) \mathcal V(t)^{-2}(\mathcal V'(t))^2\mathcal E(t)-\frac{n-2k}{n}\mathcal V(t)^{-1}\mathcal V''(t)\mathcal E(t)\bigg)
   .
\end{align*}
Now we use the assumption that $\sigma_k(g)\equiv c_g$ to evaluate the second variation at $t=0$. Writing $\bar \phi_g\coloneqq \vol_g(M)^{-1}\int_M\phi\,\mathrm d\vol_g$ gives the conformal Hessian in the inverse parametrization \begin{equation}
 \frac{ d^2}{dt^2}\mathcal F_k[g_t]\Big|_{t=0}= \frac{n-2k}{\vol_{g}(M)^{\frac {n-2k}n}}\bigg(\int_M T_{k-1} (g^{-1}A_g)(\nabla \phi,\nabla \phi)_g \, \mathrm d\vol_{g}-2kc_g\int_M(\bar \phi_g-\phi)^2\,\mathrm d \vol_{g} \bigg).\label{eq:F''0}
\end{equation} By the chain rule the Hessian of $F_{k,g}$ is given by $$Q_g(\phi)\coloneqq \left(\frac{2k}{n-2k}\right)^2 \frac{ d^2}{dt^2}\mathcal F_k[g_t]\Big|_{t=0}\,.$$
We also use $Q_g$ for its associated self-adjoint operator when writing $Q_g\phi$, its kernel, or its spectrum, for instance. As $g$ is $k$-admissible, $T_{k-1}(g^{-1}A_g)$ is positive definite and thus elliptic. Since the Euler--Lagrange equation associated to \eqref{eq:F''0} is of divergence form and $T_{k-1}(g^{-1}A_g)$ is symmetric, it is in particular self-adjoint. A self-adjoint, elliptic operator bounded from below on a compact manifold has a discrete spectrum with finite-dimensional eigenspaces, and its eigenvalues $\lambda_1\leq \lambda_2\leq \dots\leq \lambda_j\leq \dots$, counted with multiplicities, tend to infinity as $j\to \infty$; see \cite[Chapter~III, Theorem~5.8]{LawsonMichelsohn1989}, for instance, and also \cite{Garding1953Spectral, Hormander1968}. Note that $\lambda_1\geq 0$ for the Hessian at a minimizer. Moreover, on the orthogonal complement of its kernel, ellipticity and discreteness give a positive spectral gap. This did not use the assumption of non-degeneracy of $\mathcal M_k$. 

We will also need the first variation with respect to the standard parametrization $g_{u_t}$ with $u_t=1+t\phi$ given by 
\begin{equation}
    D_1F_{k,g}[\phi]=2k \frac1{\vol_g(M)^{\frac{n-2k}n}} \int_M \phi\left( \sigma_k(g)-\frac 1{\vol_g(M)}\int_M\sigma_k(g)\,\mathrm d\vol_g\right) \,\mathrm d\vol_g\,.\label{eq:mix_exp}
\end{equation}

We note that in case $G$ is not compact, then $M$ is conformally equivalent to the round sphere of some constant sectional curvature $\kappa$ by the Ferrand--Obata theorem \cite{Fer96,Oba71}, and the set of minimizers is non-degenerate. Indeed, passing to the round metric, the set of minimizers $\mathcal M_k$ consists of a normalized constant up to M\"obius transformations, the tangent space of $\R_+\mathcal M_k$ at a constant function consists of the spherical harmonics of degree $0$ and $1$, and the Hessian at a constant has a positive gap on the $L^2(g_*)$-orthogonal complement of these modes. To see the latter, we note that $g_*^{-1}A_{g_*}=\frac{\kappa}{2} \mathbbm 1$ and $$T_{k-1}(g_*^{-1}A_{g_*})=\binom{n-1}{k-1} \left(\frac {\kappa} 2\right)^{k-1}\mathbbm 1\qquad \text{and}\qquad \sigma_k(g_*^{-1}A_{g_*})=\binom{n}{k} \left(\frac {\kappa} 2\right)^{k}\,.$$ Therefore, we can write \eqref{eq:F''0} as \begin{equation*}
    \frac{d^2}{dt^2}\mathcal F_k[g_t]\Big|_{t=0}=\frac{n-2k}{\vol_{g_*}(M)^{\frac{n-2k}n}}\binom{n-1}{k-1}\left(\frac{\kappa}2\right)^{k-1}\left(\int_M |\nabla \phi|^2 \,\mathrm d\vol_{g_*}-n\kappa \int_M(\phi-\bar \phi_{g_*})^2\,\mathrm d\vol_{g_*}\right).
\end{equation*} Note that the Hessian coincides with the Hessian of the Yamabe inequality on the round sphere of curvature $\kappa$ up to a positive multiplicative constant. If we perturb along volume normalized metrics, then $\bar \phi_{g_*}=0$ holds. 

Let us now say a few words on the special case of the round sphere. As the Laplace--Beltrami operator acts on the space $\mathcal H_\ell$, $\ell\geq 0$, of spherical harmonics of degree $\ell$  via $-\Delta_{g_*}|_{\mathcal H_l}=\kappa \ell(\ell+n-1)$, the operator $-\Delta_{g_*}-n\kappa$ admits the eigenvalues $$\kappa (\ell (\ell+n-1)-n)>0$$ for $\ell\geq 2$, $0$ for $\ell=1$, and $-n\kappa$ for $\ell=0$. By the correction in terms of the mean $\bar \phi_g$ in \eqref{eq:mix_exp}, the Hessian vanishes on constants. Hence, there is a constant $C=C(n,k,\kappa)>0$ such that
\begin{equation}\label{eq:HessQ}
    Q_{g_*}(\phi)\geq C \|\phi\|_{W^{1,2}(g_*)}^2\,,\qquad \phi\in (\mathcal H_0\oplus \mathcal H_1)^\perp\,,
\end{equation} where $\perp$ denotes $L^2(g_*)$-orthogonality.

To prove non-degeneracy on the sphere, we are left to show that 
$$T_1(\R_+\mathcal M_k)=\mathcal H_0\oplus \mathcal H_1\,.$$ Differentiating a minimizer $c(1)_\Psi$  in the parameter $c$ gives $\mathcal H_0$, so we are left to prove $T_c\mathcal M_k=\mathcal H_1$ for $c$ being the volume normalization $\vol_{g_*}(\mathbb S^n)^{-1/q}$. Recall from \cite{FrankPeteranderl2024_sigma2, FPR2025} that any M\"obius transformation $\Psi$ can be written as a rotation followed by $\Psi_\xi$, $\xi\in B^{n+1}$, which is given by \begin{equation}\label{eq:Mob_param}\Psi_\xi(\omega)\coloneqq \frac{(1-|\xi|^2)\omega-2(1-\xi\cdot \omega)\xi}{1-2\xi\cdot\omega+|\xi|^2}
\end{equation}
with Jacobian \begin{equation}\label{eq:conf_Jac}
    J_{\Psi_\xi}(\omega)^{\frac 1n}=\frac{1-|\xi|^2}{1-2\xi\cdot\omega+|\xi|^2}=1+2\xi\cdot\omega +O(|\xi|^2)\,.
\end{equation}
Recalling \eqref{eq:conf} we note that $\partial_t(1)_{\Psi_{t\xi}}|_{t=0}$ is proportional to $\xi\cdot \omega$, which spans $\mathcal H_1$.

\subsection{Analytic reduction: Lyapunov--Schmidt and the {\L}ojasiewicz inequality}\label{subsec:LSL}
Let $v\in \mathcal M_k$, $g=g_v$, $m\in \N_0$, and $0<\alpha<1$. Define $$N^{m,\alpha}_g\coloneqq \left\{\phi\in C^{m,\alpha}(M):\int_M \phi \,\mathrm d\vol_g=0\right\}.$$ 

We investigate how $F_{k,g}$ behaves under perturbations of $0$ in $N_g^{2,\alpha}$. Since the numerator of $F_{k,g}[1+\phi]$ is a polynomial in $1+\phi$, $\nabla \phi$, and $\nabla^2\phi$, and its denominator is analytic as long as $1+\phi>0$, this functional is real analytic on a sufficiently small $C^{2,\alpha}(M)$-neighborhood of $0$. Similarly, its Euler--Lagrange equation is analytic as a map from $C^{2,\alpha}(M)$ to $C^{0,\alpha}(M)$; cf.~\eqref{eq:mix_exp}. Thus, the Lyapunov--Schmidt reduction is applicable; cf.~\cite{Sim83, EnNeSp}: A perturbation of a critical point is split into a second order contribution and a finite dimensional subspace. The latter is treated using an inequality by {\L}ojasiewicz \cite{Loj65}.

Denote by $$K_g\coloneqq (\ker Q_g) \cap N_g^{m,\alpha}$$ and by $K_g^\perp$ its $L^2(g)$-orthogonal complement with mean zero. Note that $K_g$ is finite-dimensional. Indeed, by elliptic regularity there is a $C>0$ such that $$\|\phi\|_{W^{2,2}(g)}\leq C(\| Q_g\phi\|_{L^2(g)}+\|\phi\|_{L^2(g)})$$ (see also \cite[p.~9]{EnNeSp}), so restricted to $K_g$ the $W^{2,2}(g)$- and $L^2(g)$-norms are equivalent, which by compactness of $W^{2,2}(g)\to L^2(g)$ implies that $K_g$ is finite-dimensional.

First, we provide the Lyapunov--Schmidt reduction.

\begin{lemma}[Lyapunov--Schmidt reduction]\label{lem:ls} Let $(M,g)$ be a closed Riemannian manifold with $g=g_v$, $v\in \mathcal M_k$, and $0<\alpha<1$. There are an open neighborhood $\mathcal U_g\subset K_g$ of $0$ and a real-analytic map $$\Phi_g:\mathcal U_g\to K_g^\perp \cap N_g^{2,\alpha}\qquad \text{with} \quad \Phi_g(0)=0\quad \text{and}\quad D\Phi_g(0)=0$$ such that both the graph and the reduced functional given by $$\alpha_g(\xi)\coloneqq 1+\xi+\Phi_g(\xi) \qquad  \text{and} \qquad \Sigma_g(\xi)\coloneqq F_{k,g}[\alpha_g(\xi)]$$ are real analytic in $\xi\in \mathcal U_g$ and satisfy  
 \begin{equation}
    \label{eq:LS_red_der}D_{\alpha_g(\xi)}F_{k,g}[\zeta]=D_{\xi}\Sigma_g[\pi_{K_g}\zeta]\qquad \text{for all} \ \ \zeta\in N_g^{2,\alpha}\,,
    \end{equation} where $\pi_{K_g}$ is the $L^2(g)$-orthogonal projection onto $K_g$. In particular, the first variation vanishes on $K_g^\perp\cap N_g^{2,\alpha}$. Moreover, $\alpha_g(\xi)$ is positive and $k$-admissible, and every critical point sufficiently close to $1$ (up to multiplication by a constant) can be represented in the form $\alpha_g(\xi)$ for some $\xi\in \mathcal U_g$. In a sufficiently small neighborhood, $\alpha_g(\xi)$ and its derivative are bounded in $C^{m,\alpha}$ for every fixed $m\in \N$.
\end{lemma}

Since $K_g$ is finite-dimensional, we do not have to specify in which norm we have to be sufficiently close to $1$ as all norms are equivalent. In Lemma \ref{lem:ls} the neighborhoods of critical points can be taken in $C^{2,\alpha}(M)$.

This is the analogue of \cite[Lemma 2.2 (1)]{EnNeSp}. For the sake of completeness, we give the standard reduction argument on the set of mean zero functions $N_g^{2,\alpha}$; cf.~\cite[Section~2.3]{FM20}.

\begin{proof}
    On $1+N_g^{2,\alpha}$, the first variation defines an analytic map from a neighborhood of $0$ in $N_g^{2,\alpha}$ to $N_g^{0,\alpha}$. Evaluated at $1+\xi$, its linearization at $\xi$ is given by the associated self-adjoint operator of the Hessian $Q_g$. Its restriction from $K_g^\perp \cap N_g^{2,\alpha} $ to $K_g^\perp \cap N_g^{0,\alpha}$ is an isomorphism. An analytic version of the implicit function theorem then provides the map $\Phi_g$ with $\Phi_g(0)=0$ and $D_0\Phi_g=0$. 

    By construction, the first variation of $F_{k,g}$ at $\alpha_g(\xi)$ vanishes on $K_g^\perp\cap N_g^{2,\alpha}$. Using that the range of $D_\xi \Phi_g$ lies in $K_g^\perp\cap N_g^{2,\alpha}$, where $D_{\alpha_g(\xi)}F_{k,g}$ vanishes, and the chain rule, we find $$D_{\xi}\Sigma_g[\pi_{K_g}\zeta]=D_{\alpha_g(\xi)} F_{k,g}[\pi_{K_g}\zeta+D_\xi \Phi_g[\pi_{K_g}\zeta]]=D_{\alpha_g(\xi)} F_{k,g}[\zeta]$$ for every $\zeta\in N_g^{2,\alpha}$. This proves \eqref{eq:LS_red_der}. Local uniqueness in the implicit function theorem provides the representation of every critical point sufficiently close to $1$ in $C^{2,\alpha}$ after adjusting a scale, so the functions lie in $1+N_g^{2,\alpha}$. Positivity and $k$-admissibility follow by shrinking $\mathcal U_g$. Applying the same argument from $N^{m+2}_g$ to $N^{m}_g$ and local uniqueness (and shrinking $\mathcal U_g$ further if necessary), we obtain $C^{m,\alpha}$-bounds for every $m>0$ and $0< \alpha<1$.
\end{proof}

    Now we apply the following inequality by {\L}ojasiewicz \cite{Loj65}; see also \cite[Lemma 3.2]{EnNeSp}. 
    If $\Sigma:\R^m\to \R$ is a real-analytic function near $0$ with $\Sigma\geq 0$ and $\Sigma(0)=0$, then there exist $\tilde \delta>0$, $\tilde c>0$, and $\gamma\geq 0$ such that for all $\phi\in B_{\tilde \delta}(0)$ we have 
    \begin{equation}
        \Sigma(\phi)\geq \tilde c\dist(\phi,\mathcal Z)^{2+\gamma}\,,\label{eq:low}
    \end{equation}  where $\mathcal Z\coloneqq \{\bar \phi\in B_{\tilde \delta}(0): \Sigma(\bar \phi)=0\}$, $\dist(\phi,\mathcal Z)\coloneqq \inf_{\bar\phi\in \mathcal Z} |\phi-\bar\phi|$ denotes the Euclidean distance, and the closed ball  lies in the domain of analyticity. 
    Note that $D_0\Sigma=0$ as used in \cite[Lemma 3.2]{EnNeSp} follows from $\Sigma\geq 0$ and $\Sigma(0)=0$. Setting $\Sigma=\Sigma_g-Y_k$ satisfies the required conditions. The constants and exponents depend on the fixed chart.

    Note that a non-degenerate $\mathcal M_k$ gives $\Sigma_g\equiv Y_k$ close to $0$, and thus $\gamma=0$. Indeed, in the non-degenerate case, the minimizers with mean one form a smooth manifold with tangent space $K_g$. As projection onto $K_g$ is a local diffeomorphism, Lemma \ref{lem:ls} implies that points on the graph near $0$ (which have mean one) are minimizers. On the orthogonal directions, perturbations of $0$ are controlled by the quadratic Hessian spectral gap, so $\gamma=0$.

\subsection{Compactness of almost-minimizers}

In this subsection, we gather some compactness results that we need later.

\begin{proposition} \label{prop:cpct}Let $2\leq k< n/2$ and $(M^n,[g_0])$ be a smooth, closed, connected, locally conformally flat Riemannian manifold with $k$-admissible metric $g_0$. The following two facts hold.
\begin{enumerate}
    \item[(i)] $Y_k(M,[g_0])>0$ and $\mathcal M_k$ is nonempty.
    \item[(ii)] If $(h_j)\subset \mathcal C_k([g_0])$ is a sequence of normalized functions with
    $$\sigma_k(g_{h_j})\equiv c_j\qquad \text{and}\qquad F_{k,g_0}[h_j]\to Y_k(M,[g_0])\,,$$ then after possibly passing to a subsequence there are normalized $v\in \mathcal M_k$ and $(\Psi_j)\subset G$ such that $(h_j)_{\Psi_j}\to v$ in any $C^m$-norm, $m\in \N$.
\end{enumerate}
    
\end{proposition}

\begin{proof}
   We start with a proof of positivity of $Y_k$. If $M$ is orientable, this follows from \cite[Corollary 1]{GW04}.  Otherwise, we consider the orientable double cover $\pi:\tilde M\to M$ of $M$. Then $A_{\pi^*g}=\pi^*A_g$ and the integration over $\tilde M$ doubles the integrals over $M$. Hence, we have $$\mathcal F_k[\pi^*h]=2^{\frac{2k}n} \mathcal F_k[h]$$ for any conformal metric $h$ on $M$, so the orientable case gives $$0<Y_k( \tilde M,[\pi^*g_0])\leq2^{\frac{2k}n}Y_k( M, [g_0])\,.$$

   To prove that $\mathcal M_k \neq \emptyset$, we consider a flow introduced by Guan and Wang \cite{GuanWangFlow2003}; for more details, we refer to Subsection \ref{subsec:GW}. Choose an $L^q(g_0)$-normalized minimizing sequence $(f_j)\subset \mathcal C_k([g_0])$, so $F_{k,g_0}[f_j]\to Y_k$, and start the Guan--Wang flow at each $g_{f_j}$. By \cite[Theorem 1]{GuanWangFlow2003} and \cite[Theorem 2]{GW04}, the flow $g_{f_j(t)}$ preserves volume and $k$-admissibility, decreases $\mathcal F_{k}$, and converges to a smooth, constant $\sigma_k$-curvature metric $g_{h_j}$ with $\|h_j\|_{L^q(g_0)}=1$. Moreover, the time-dependent metrics $g_{f_j(t)}$ converge in $C^{4,\alpha}(M)$, $0<\alpha<1$, and their $\sigma_k$-curvatures in $L^2(g_0)$. The flow also applies to non-orientable $M$; cf.~Subsection \ref{subsec:GW}.  Hence, in the limit $j\to\infty$, we obtain $$Y_k\leq F_{k,g_0}[h_j]\leq F_{k,g_0}[f_j]\to Y_k\,.$$ 
   
   If $M$ is not conformally equivalent to the sphere, we set $$\beta_j\coloneqq F_{k,g_0}[h_j]^{\frac{n-2k}{4k^2}}\,,$$ where $F_{k,g_0}[h_j]= \sigma_k(g_{h_j})$ by volume normalization and constant $\sigma_k$-curvature equation for $h_j$. Note that $F_{k,g_0}[h_j]\to Y_k>0$ as $j\to \infty$, so $\beta_j$ stays positive in the limit. Then $\sigma_k(g_{\beta_j h_j})=1$ and  the uniform upper and lower $C^{m,\alpha}$-bounds by Li--Li \cite[Theorem 1.1', Remark 1.4]{Li2005} are applicable for every $m\in \N$, $0<\alpha<1$. After passing to a subsequence and scaling back, the sequence $(h_j)$ converges smoothly to a normalized, $k$-admissible minimizer, which concludes (i) in case of manifolds not conformally equivalent to the sphere.

If $M$ is conformally equivalent to the sphere, we pass to the round metric. Then by Viaclovsky's classification of critical points for space forms \cite[Theorem 3]{Via00} every normalized solution $h_j$ of the constant $\sigma_k$-curvature equation is the same constant up to a M\"obius transformation $\Phi_j$. Since all these solutions have the same energy and $F_{k,g_0}$ is decreasing, this proves that all $h_j$ are minimizers, which concludes (i).

   Next, we prove (ii). If $M$ is not conformally equivalent to the sphere, we can apply the previous argument using the uniform Li--Li estimates to the given sequence $(h_j)$ and conclude by taking $\Psi_j=\mathbbm 1$. If $M$ is conformally equivalent to the sphere, then we work in the round background metric and the same classification argument for constant $\sigma_k$-curvature solutions provides $\Psi_j$ for which $(h_j)_{\Psi_j}$ is the constant for all $j$ by normalization. This proves the claim. 
\end{proof}

\subsection{A unified notion of distance} We consider the distance
\begin{equation}
    \label{eq:dist_def}\dist_{k,\gamma}(u)\coloneqq \inf_{\lambda>0,\Psi\in G\,,  v\in \mathcal M_k} (\|\lambda v^{-1}(u)_\Psi-1\|^{2+\gamma}_{W^{1,2}(g_v)}+\|\lambda v^{-1}(u)_\Psi-1\|^{\max\{2k,2+\gamma\}}_{W^{1,2k}(g_v)})\,.
\end{equation} It is bounded since taking $\lambda\to 0$ provides the upper bound \begin{equation}
    \dist_{k,\gamma}(u)\leq \|1\|_{W^{1,2}(g_v)}^{2+\gamma}+\|1\|_{W^{1,2k}(g_v)}^{\max\{2k,2+\gamma\}}=2\label{eq:dist_bound} 
\end{equation}as the volume of $v\in \mathcal M_k$ is normalized. Moreover, $\dist_{k,\gamma}(u)$ comprises both the case of compact and non-compact groups of conformal diffeomorphisms. Recall that in the latter case $M$ is conformally equivalent to the sphere by the Ferrand--Obata theorem \cite{Fer96,Oba71}. The next lemma tells us how $\dist_{k,\gamma}$ simplifies in each case.

\begin{lemma} \label{lem:dist_notions}
 (i) If $G$ is compact, then $\dist_{k,\gamma}(u)$ is equivalent to $$d_{k,1}(u)\coloneqq \inf_{\lambda>0, v\in \mathcal M_k} (\|\lambda u-v\|^{2+\gamma}_{W^{1,2}(g)}+\|\lambda u-v\|^{\max\{2k,2+\gamma\}}_{W^{1,2k}(g)})\,.$$\\
(ii)  If $M$ is the round unit sphere, then  $\dist_{k,\gamma}(u)$ is equivalent to $$d_{k,2}(u)\coloneqq\inf_{\lambda>0, \Psi\in G} (\|\lambda (u)_\Psi-v\|^{2+\gamma}_{W^{1,2}(g_*)}+\|\lambda  (u)_\Psi-v\|^{\max\{2k,2+\gamma\}}_{W^{1,2k}(g_*)})$$ with $v\equiv |\mathbb S^n|^{-1/q}$ and $G$ the set of M\"obius transformations on $\mathbb S^n$.
\end{lemma}

The equivalence constants depend on $M,[g],k,\gamma$ in (i) and only on $n,k,\gamma$ in (ii).  These distances should be compared with the ones in \eqref{eq:sigma_2} and \eqref{eq:Yamabe_1}.

\begin{proof} (i) Let $G$ be compact. Together with compactness of $\mathcal M_k$ up to conformal diffeomorphisms, $\mathcal M_k$ is compact in every $C^m$, $m\in \N$. Therefore, we have $$\inf_{v\in\mathcal M_k}\inf_M v>0 \qquad \text{and} \qquad \sup_{v\in \mathcal M_k}\|v\|_{C^m}<\infty\,.$$ As a consequence, $\|\cdot\|_{W^{1,p}(g)}$ and $\|\cdot/v\|_{W^{1,p}(g_v)}$ are equivalent, $p\in \{2,2k\}$. Thus, inserting $\Psi=\mathbbm 1$ in $\dist_{k,\gamma}(u)$, there is some $C>0$ such that $$\dist_{k,\gamma}(u)\leq C d_{k,1}(u)\,.$$ By compactness of $G$, also $\|(\,\cdot\,)_\Psi\|_{W^{1,p}(g)}$ and $\|\cdot\|_{W^{1,p}(g)}$ are equivalent, $p\in \{2,2k\}$. Hence, there is some $C>0$ such that $$\dist_{k,\gamma}(u)\geq C d_{k,1}(u)\,.$$

(ii) Now let $M$ be a round unit sphere. By Viaclovsky's classification \cite{Via00}, every normalized minimizer is of the form $v=c(1)_\Phi$, $\Phi\in G$, $c=|\mathbb S^n|^{-1/q}$. Thus, $g_v=|\mathbb S^n|^{-2/n}\Phi^*(g_*)$ and 
$$\|\lambda (u)_\Psi v^{-1}-1\|^p_{W^{1,p}(g_v)}= c^{q-p}\|\lambda (u)_{\Psi\circ \Phi^{-1}}-c\|^p_{L^p(g_*)}+c^{q-p-\frac{pq}n}\lambda\|\nabla (u)_{\Psi\circ \Phi^{-1}}\|^p_{L^p(g_*)}$$ for $p\in\{2,2k\}$. Taking infima on both sides proves the claim. 
\end{proof}

Next, we describe properties of $\dist_{k,\gamma}$.

\begin{lemma}\label{lem:att}
    There is a $\delta^*>0$ such that for every $L^q(g_0)$-normalized $u \in W^{1,2k}(g_0)$ with $\dist_{k,\gamma}(u)<\delta^*$ the infimum in $\dist_{k,\gamma}(u)$ is attained. The map $\min\{\dist_{k,\gamma}(u),\delta^*\}$ is continuous on the set of $L^q(g_0)$-normalized functions with respect to $u\in W^{1,2k}(g_0)$.
\end{lemma}

\begin{proof}
   Fix $\gamma\geq 0$. All constants in this lemma may depend on $M,[g_0],k,\gamma$. The parameters in $\dist_{k,\gamma}(u)$ are $(\lambda,\Psi,v)$. First, we explain that the scale parameter $\lambda$ stays in a compact set if $\dist_{k,\gamma}(u)$ is sufficiently small. For $L^{q}$-normalized $u$ and $v$, the 
    triangle inequality gives $$|\lambda-1|\leq \|\lambda (u)_\Psi v^{-1}-1\|_{L^{q}(g_v)}\,.$$ 
    If $\dist_{k,\gamma}(u)<\delta^*$, then we can choose $(\lambda,\Psi, v)$ such that $$\|\lambda v^{-1}(u)_\Psi-1\|^{2+\gamma}_{W^{1,2}(g_v)}+\|\lambda v^{-1}(u)_\Psi-1\|^{\max\{2k,2+\gamma\}}_{W^{1,2k}(g_v)}< \delta^*\,.$$ We deduce that $$|\lambda-1|^{\max\{2k,2+\gamma\}}\leq C\delta^*$$ with a $C>0$ from Sobolev embedding, so choosing $\delta^*$ small enough, we can assume $\lambda\in [1/2,3/2]$. Note that $C>0$ can be chosen uniformly in $v\in\mathcal M_k$ by Proposition \ref{prop:cpct}, (ii).

    If $G$ is compact, then Proposition \ref{prop:cpct}, (ii), gives compactness of $\mathcal M_k$ up to symmetries. Since the infimum is taken over the compact sets $G$ and $\mathcal M_k$ and the expression minimized in $\dist_{k,\gamma}$ is jointly continuous in $(u,\lambda,\Psi,v)$, the infimum in $\dist_{k,\gamma}(u)$ is attained for some parameters.

    If $G$ is not compact, then we pass to the round background metric by Ferrand--Obata theorem \cite{Fer96,Oba71},  and the minimizers $\mathcal M_k$ are given by $c(1)_{\Phi}$, $\Phi\in G$, with $c>0$ determined by volume normalization. Hence, the distance can equivalently be written as the infimum over $(\lambda, \Theta)$ with $\Theta\coloneqq \Psi\circ\Phi^{-1}$; cf.~Lemma \ref{lem:dist_notions}, (ii). Recall that a M\"obius transformation $\Theta$ can be written as a rotation followed by a $\Theta_\xi$, $\xi\in B^{n+1}$, as defined in \eqref{eq:Mob_param}. 

    The only noncompact scenario is $|\xi|\to 1$, and in this case the Jacobian $J_{\Theta_\xi}=(1)_{\Psi_\xi}^q$ given in  \eqref{eq:conf_Jac} satisfies $J_{\Theta_\xi}\to 0$ pointwise almost everywhere. Take sequences $|\xi_j|\to 1$ and $\lambda_j\to \lambda\in [1/2,3/2]$ and set $\Theta_j=\Theta_{\xi_j}$. Set $c=\vol_{g_0}(M)^{-1/q}$. By conformal invariance of the $L^q$-norm and the Brezis--Lieb lemma, we find 
    \begin{equation}\label{eq:u_j}
        \|\lambda_j(u_j)_{\Theta_j}-c\|_{L^q(g_0)}^q=\|\lambda_ju_j-c(1)_{\Theta_j^{-1}}\|_{L^q(g_0)}^q\to \lambda^q +1\geq 2^{-q}+1
    \end{equation} as $j\to\infty$ for every sequence $(u_j)$ of smooth, volume-normalized functions approximating $u$ in $L^q(g_0)$; cf.~\cite[Section 3]{FrankPeteranderl2024_sigma2}. Since the right side of \eqref{eq:u_j} is bounded uniformly from below, the parameter $\xi$ has to stay in a compact set $K'\subset B^{n+1}$ if $\delta^*$ is chosen sufficiently small. Hence, together with $\lambda$ staying in a compact set, we know that the infimum can be attained.

    The infimum of continuous functions is upper semicontinuous. On the other hand, if $u_j\to u$ with $\liminf_{j\to\infty} \dist_{k,\gamma}(u_j)<\delta^*$, choose a subsequence realizing this $\liminf$ and parameters in $K'$ for which the expression that is minimized in $\dist_{k,\gamma}$ differs from $\dist_{k,\gamma}(u_j)$ by $1/j$. By compactness of $K'$, we find a convergent subsequence. Since the expression that is minimized in $\dist_{k,\gamma}$ is jointly continuous in the parameters, the limit of the parameters is admissible for the infimum $\dist_{k,\gamma}(u)$. Therefore, we have $$\dist_{k,\gamma}(u)\leq \liminf_{j\to\infty}\dist_{k,\gamma}(u_j)\,.$$ This proves continuity of $\dist_{k,\gamma}$ when restricting to $u$ that satisfy $\dist_{k,\gamma}(u)<\delta^*$.
\end{proof}

\subsection{A global two-term bound} The following proposition is the main coercivity bound used for the high frequency estimate in Section \ref{sec:3}. Its proof does not depend on the local stability argument, does not require connectedness, and allows the top-gradient constant to be independent of the chosen background metric. Instead it uses a property of the superadditivity of the Newton tensor. Since we did not find a reference for this statement, we give a short proof here. It is a consequence of polarizing the $k$-th elementary symmetric polynomials.
\begin{lemma}[Superadditivity of the Newton tensor]\label{lem:superadd} If $A,B\in \Gamma_k^+$, $k\geq 2$, are symmetric matrices, then 
$$T_{k-1}(A+B)\geq T_{k-1}(A)+T_{k-1}(B)\,.$$
\end{lemma}
Note that for $k=2$, the Newton tensor $T_1$ is linear. For $k=1$, $T_0=\mathbbm 1$, so superadditivity does not hold.
\begin{proof}
    The proof is based on polarization of $k$-th elementary symmetric polynomials as can be found in \cite[p.~964f.]{Garding1959} or \cite[p.~88]{CaseWang2018}, for instance. If $A_i$, $i\in \{1,\dots,k\}$, are symmetric matrices, then $$\sigma_k(A_1,\dots,A_k)\coloneqq \frac 1{k!} \frac{\partial^k}{\partial t_1\dots \partial t_k}\sigma_k(t_1A_1+\dots+t_k A_k)|_{t_1=\dots=t_k=0}\,.$$ In particular, $\sigma_k(A,\dots,A)=\sigma_k(A)$. Since $$D\sigma_k(C)[H]=k\sigma_k(C,\dots,C,H)=\tr(T_{k-1}(C)H)\,,$$ we can evaluate the Newton tensor at a specific state $\rho\in \R^n$ by setting $H=\rho\otimes \rho$. Then we obtain by multilinearity that $$\langle \rho, T_{k-1}(A+B)\rho\rangle =k\sigma_k(A+B,\dots,A+B, \rho\otimes\rho)=k\sum_{j=0}^{k-1}\binom{k-1}j \sigma_k(A^{(j)},B^{(k-1-j)},\rho\otimes\rho)\,,$$ where $C^{(m)}$ denotes the vector of length $m$ with entries $C$. Note that $$\sigma_k(A^{(j)},B^{(k-1-j)},(\rho\otimes\rho)+\varepsilon \mathbbm 1)>0$$ for every $\varepsilon>0$ by \cite[Theorem 5]{Garding1959} as all entries are in $\Gamma_k^+$. Taking the limit $\varepsilon\to 0$ still gives nonnegativity. Hence, we conclude \begin{align*}
        \langle \rho, T_{k-1}(A+B)\rho\rangle&\geq k\left(\sigma_k(A,\dots, A,\rho\otimes\rho)+\sigma_k(B,\dots,B,\rho\otimes\rho)\right)\\&= \langle \rho, T_{k-1}(A)\rho\rangle+\langle \rho, T_{k-1}(B)\rho\rangle\,,
    \end{align*}which finishes the proof as $\rho$ was arbitrary.
\end{proof}

With this property, we can now state and prove our main coercivity bound. To this end, we denote
\begin{equation*}
E_{k,g}[V]\coloneqq \int_M V^{-n}\sigma_k(B_g(V))\,\mathrm d \vol_g\,.\end{equation*}

\begin{proposition}[Global second order bound]
\label{prop:glob_W12k}
    Let $2\leq k<n/2$ and $(M^n,[g])$ be a smooth, closed, and locally conformally flat Riemannian manifold with $k$-admissible metric $g$. If $u\in \mathcal C_k([g])$ and $V=u^{-q/n}$, then there are  $c_{n,k}, C_{g}>0$ such that
    \begin{equation}\label{eq:global_bound}
    E_{k,g}[V]- E_{k,g}[1]-D_1E_{k,g}[V-1]\geq  c_{n,k} \left(\int_M|\nabla u|^{2k}\,\mathrm d \vol_g+C_{g}\int_M|\nabla u|^{2}\,\mathrm d \vol_g\right).
    \end{equation}
   \end{proposition}
For fixed $n,k$, $C_g$ depends only on the minimal eigenvalue of $T_{k-1}(g^{-1}A_g)$. If $g$ ranges over a compact smooth family contained in $\Gamma_k^+$, $C_{g}$ can be chosen uniformly in $g$.

\begin{proof}
In this proof, we use the parametrization $V=u^{-2k/(n-2k)}$. In the formulation \eqref{eq:Schouten_V}, the total $\sigma_k$-curvature is given by $$\int_M  \sigma_k(g_u^{-1}A_{g_u})\,\mathrm d\vol_{g_u}=\int_M V^{-n} \sigma_k(B_g(V))\,\mathrm d\vol_g\,.$$ 
 We consider the affine interpolation between $g$ and $V^{-2}g$ given by
\begin{equation*}
    g_t\coloneqq V_t^{-2}g\qquad \text{with} \quad V_t\coloneqq 1+t(V-1)>0 \,,\quad 0\leq t\leq1\,.
\end{equation*}
The whole path $g_t$ is $k$-admissible. Indeed, the Schouten tensor for $V_t$ is \begin{equation*}
    B_g(V_t)=V_t(1-t) g^{-1}A_g+\frac{tV_t}{V} B_g(V)+\frac{t(1-t)}{2V} |\nabla V|^2\mathbbm 1\,,
\end{equation*} which is a linear combination of $g^{-1}A_g,B_g(V), \mathbbm 1$ in the convex cone $\Gamma_k^+$ with nonnegative coefficients.
Since $V_t$ is of the form $1+t\phi$ as in \eqref{eq:Vt} with $\phi=V-1$, we can use the computations from Subsection \ref{subsec:conf_Hess} in terms of $$\mathcal E(t)=\int_M\sigma_k(g_t)\,\mathrm d\vol_{g_t}\qquad \text{and}\qquad  \psi_t=\frac{V-1}{V_t}\,.$$ Hence, the second variation formula 
\eqref{eq:E''} gives $$
\mathcal E''(t)\geq (n-2k)
\int_M T_{k-1} (B_g(V_t))(\nabla_{g_t} \psi_t,\nabla_{g_t} \psi_t)_{g_t} \, \mathrm d\vol_{g_t}\,.$$ where we used the nonnegativity of $\sigma_k(g_t)$. Integrating $\mathcal E''(t)$ then yields
\begin{align}\notag
    \mathcal E(1)-\mathcal E(0)-\mathcal E'(0)&=\int_0^1 (1-t)\mathcal E''(t)\,\mathrm dt\\&\geq (n-2k)
\int_0^1(1-t)\int_M T_{k-1} (B_g(V_t))(\nabla_{g_t} \psi_t,\nabla_{g_t} \psi_t)_{g_t} \, \mathrm d\vol_{g_t}\,\mathrm d t\,.\label{eq:E1E0}
\end{align} 

To extract the top and second order gradient contribution from the Newton tensor, we write $B_g(V_t)=X+Y+s\mathbbm 1$ with $X\coloneqq V_t(1-t) g^{-1}A_g$, $Y\coloneqq tV_t B_g(V)/V$, and $s\coloneqq t(1-t)|\nabla V|^2/2V$ and apply superadditivity of the Newton tensor; see Lemma \ref{lem:superadd}. Since $T_{k-1}(Y)\geq 0$, $T_{k-1}$ is $(k-1)$-homogeneous, and $T_{k-1}(\mathbbm 1)=\binom{n-1}{k-1}\mathbbm 1$, superadditivity yields
\begin{equation}\label{eq:T_bdd}
    T_{k-1}(X+Y+s\mathbbm 1)\geq T_{k-1}(X)+\binom{n-1}{k-1}s^{k-1}\mathbbm 1\,.
\end{equation} 
If we use $$\nabla_g \psi_t=V_t^{-2}\nabla_g V\,,\qquad  \mathrm d\vol_{g_t}=V_t^{-n}\mathrm d\vol_{g}\,,\qquad \text{and}\qquad|\nabla_{g_t} \psi_t|^2_{g_t}=V_t^{-2}|\nabla_g V|_g^2$$ and denote by $\mu_g>0$ a lower bound on the bottom of the spectrum of $T_{k-1}(g^{-1}A_g)$ that is  uniform in the basepoint in $M$, then \eqref{eq:E1E0} and \eqref{eq:T_bdd} imply
\begin{align}
\notag \mathcal E(1)-\mathcal E(0)-\mathcal E'(0)\geq \ & (n-2k)\mu_g\int_M |\nabla V|_g^{2} \int_0^1(1-t)^k V_t^{k-n-3}\,\mathrm d t\, \mathrm d\vol_{g}
\\ & +\frac{n-2k}{2^{k-1}}\binom{n-1}{k-1} \int_M |\nabla V|_g^{2k} V^{-(k-1)}\int_0^1t^{k-1}(1-t)^k V_t^{-n-2}\,\mathrm d t\, \mathrm d\vol_{g}\,.\label{eq:E1E0V}
\end{align} 

Since $\mathcal E(1)-\mathcal E(0)-\mathcal E'(0)$ is the left side of \eqref{eq:global_bound}, it remains to prove that the right side of \eqref{eq:E1E0V} is bounded from below by the right side of \eqref{eq:global_bound}. For this purpose, it suffices to prove that there is a $c_{n,k}>0$ such that the following elementary bounds for the one-dimensional inner integrals in \eqref{eq:E1E0V} given by
\begin{align}\label{eq:one_int1}
\int_0^1 (1-t)^k (1-t(1-s))^{k-n-3}\,\mathrm dt&\geq c_{n,k}s^{-\frac nk}\,,\\\label{eq:one_int2}
\int_0^1 t^{k-1}(1-t)^k (1-t(1-s))^{-n-2}\,\mathrm dt&\geq c_{n,k}s^{k-n-1}
\end{align} hold, where $s=V(x)$ is evaluated pointwise at $x\in M$. 

We prove both estimates by distinguishing three different cases of $s$. 

If $0<s\leq 1/4$, then we bound the integral from below by its restriction to $1-2s\leq t\leq 1-s$. Then $t\geq 1/2$, $1-t\geq s$, and $1-t(1-s)\leq 3s$, so \begin{align*}
    \int_{1-2s}^{1-s} (1-t)^k (1-t(1-s))^{k-n-3}\,\mathrm dt&\geq 
    3^{k-n-3}s^{2k-n-2}\geq 3^{k-n-3}s^{-\frac nk}\,,\\\int_{1-2s}^{1-s} t^{k-1}(1-t)^k (1-t(1-s))^{-n-2}\,\mathrm dt&\geq 2^{-(k-1)} 3^{-n-2}s^{k-n-1}\,.
\end{align*} Note that $k-n-3<0$ and $2k-n-2=-n/k-(k-1)(n-2k)/k\leq -n/k$. 

If $s\geq 4$, then we bound the integral from below by its restriction to $0\leq t\leq s^{-1}$. Then $1-t\geq 3/4$ and $1-t(1-s)\leq 2$, so \begin{align*}
\int_{0}^{s^{-1}} (1-t)^k (1-t(1-s))^{k-n-3}\,\mathrm dt&\geq  \left(\frac34\right)^k 2^{k-n-3}s^{-1}\geq \left(\frac34\right)^k 2^{k-n-3}s^{-\frac nk} \,,\\
    \int_{0}^{s^{-1}} t^{k-1}(1-t)^k (1-t(1-s))^{-n-2}\,\mathrm dt&\geq  \left(\frac34\right)^k 2^{-n-2}\int_0^{s^{-1}}t^{k-1}\,\mathrm dt\geq \left(\frac34\right)^k 2^{-n-2}\frac 1ks^{k-n-1} \,.
\end{align*} Note that we used $n\geq 2k$ here. 

Finally, if $1/4\leq s\leq 4$, then the quotients of both sides of the respective inequality are each a continuous positive function in $s$ on a compact interval, so they are bounded from below by some constant. This concludes \eqref{eq:one_int1} and \eqref{eq:one_int2}, and therefore \eqref{eq:global_bound} after reparametrization of the conformal factor $V^{-2}=u^{4k/(n-2k)}$.
\end{proof}

\section{Local stability} 
\label{sec:3}
In this section, we prove the local part, Proposition \ref{prop:loc}. To this end, we first explain how to decompose the spectrum, which is essential to derive the local bound.

\subsection{Spectral decomposition of almost minimizers} On the round sphere, the minimizer is constant up to M\"obius transformations, so there is only one fixed Hessian; cf.~\eqref{eq:HessQ}. Its associated self-adjoint second-order operator is (up to a positive, multiplicative constant) given by $$-\Delta_{g_*}-n\kappa_*(\mathbbm 1-\Pi_0)\,,$$ where $\Pi_0$ is the orthogonal projection onto the constant functions; cf.~Subsection \ref{subsec:conf_Hess}. Thus, the above decomposition on the sphere corresponds to the standard decomposition into spherical harmonics, where the kernel of the Hessian is generated by  multiplication by constants and M\"obius transformations. More precisely, we have 
$$L^2(\mathbb S^n)=\mathcal H_0\oplus \mathcal H_1\oplus E^{\me}_{\Lambda,1}\oplus E^{\hi}_{\Lambda,1}\qquad \text{with}\quad  E^{\me}_{\Lambda,1}=\bigoplus_{\ell\geq 2: \mu_\ell\leq \Lambda }\mathcal H_\ell\quad \text{and}\quad E_{\Lambda,1}^{\hi}=\bigoplus_{\ell\geq 2: \mu_\ell> \Lambda }\mathcal H_\ell\,,$$ where $\Lambda>0$ is a spectral cutoff, $\mathcal H_\ell$ the space of spherical harmonics of degree $\ell$, and $\mu_\ell$ the corresponding eigenvalue. This is the setting used in \cite{FrankPeteranderl2024_sigma2}.

On general locally conformally flat manifolds, the minimizers need not be unique (up to conformal diffeomorphisms), and the kernel need not be generated by a smooth family of minimizers. We therefore fix the spectral decomposition at the center of each Lyapunov--Schmidt chart. To this end, we record the following finite reduction principle.

\begin{lemma}[Reduction to one neighborhood]\label{lem:one_nghb}
 Let $(M,[g])$ not be conformally equivalent to the round sphere. For every $v_0\in \mathcal M_k$, let $\mathcal U(v_0)\subset \mathcal M_k$ denote a neighborhood of $v_0$. Then there are finitely many minimizers $v^{(1)},\dots, v^{(N)}\in \mathcal M_k$ with $N\in \N$ and neighborhoods \begin{equation*}
     \mathcal U_i\coloneqq \mathcal U(v^{(i)})\,,\qquad i\in \{1,\dots,N\}\,,\end{equation*} such that for every $v\in \mathcal M_k$ there are $\Phi\in G$ and $i\in \{1,\dots,N\}$ with \begin{equation}
         \label{eq:fin_cpct_nghb}(v)_\Phi\in \mathcal U_i\,.\end{equation} Moreover, for $1\leq s<\infty$ we have \begin{equation*}
             \|(u)_\Phi (v)_\Phi^{-1}-1\|_{W^{1,s}(g_{(v)_\Phi})}=\|u v^{-1}-1\|_{W^{1,s}(g_{(v)_\Phi})}.
         \end{equation*}
\end{lemma}

As a consequence, after a conformal diffeomorphism and taking a subsequence, any local argument can be carried out in one fixed neighborhood as the background minimizer lies in one of finitely many fixed neighborhoods. All estimates thereafter can be taken uniformly in these neighborhoods.

\begin{proof}
    If $\pi:\mathcal M_k\to \mathcal M_k/G$ denotes the quotient map, then $\{\pi(\mathcal U(\tilde v)):\tilde v\in \mathcal M_k\}$ is an open cover of $\mathcal M_k/G$. By compactness of $\mathcal M_k/G$, there are finitely many $v^{(1)},\dots,v^{(N)}\in  \mathcal M_k$ such that $\pi(\mathcal U_i)$ with $\mathcal U_i=\mathcal U(v^{(i)})$ covers $\mathcal M_k/G$. Hence, there exists a $\Phi\in G$ and an $i$ such that \eqref{eq:fin_cpct_nghb} holds. 
    Note that applying a conformal diffeomorphism $\Phi$ does not change the distance to some $w\in \mathcal C_k([g])$ as \begin{equation*}
        \|(w)_\Phi (v)_\Phi^{-1}-1\|_{W^{1,p}(g_{(v)_\Phi})}=\|(w v^{-1}-1)\circ \Phi\|_{W^{1,p}(\Phi^*g_{v})}=\|w v^{-1}-1\|_{W^{1,p}(g_{v})}
    \end{equation*} for all $1\leq p< \infty$. 
\end{proof}

For the rest of the local argument, let $v\in \mathcal M_k$ with $g=g_v$. Recall from Subsection \ref{subsec:LSL} that $K_g$ is the mean zero kernel of the Hessian $Q_g$, $\alpha_g$ is the graph, and $\Sigma_g$ denote the reduced functional from Lemma \ref{lem:ls}. Choose a $\Lambda>0$ that is not in the spectrum of $Q_g$.

We further denote by $E^{\me}_{\Lambda,g}$ the direct sum of all eigenspaces of $Q_{g}$ with positive eigenvalues below $\Lambda>0$ and by $E^{\hi}_{\Lambda, g}$ the orthogonal complement of $\R1\oplus K_g\oplus E^{\me}_{\Lambda, g}$. Therefore, we have the decomposition 
\begin{equation}\label{eq:space_decomp}
    L^2(g)=\R 1\oplus K_g\oplus E^{\me}_{\Lambda,\hat v}\oplus E^{\hi}_{\Lambda,\hat v}\,.
\end{equation}
All orthogonal complements and projections in this section are taken with respect to $L^2(g)$. 
Note that all norms on finite-dimensional spaces are equivalent and by elliptic regularity, the eigenfunctions for finitely many eigenvalues are smooth; in particular we have 
\begin{equation*}
    \|z\|_{C^m}\leq C_{m,\Lambda} \|z\|_{W^{1,2}(g)}\,,\qquad z\in K_g\oplus E_{\Lambda,g}^\me\,,\ m\in \N\,.
\end{equation*}
As $g$ is minimizing, its Hessian is nonnegative, so by ellipticity we obtain a spectral gap on $K_g^\perp$ inside $N_g^{2,\alpha}$; cf.~Subsection \ref{subsec:LSL}. Note that for the finitely many neighborhoods constructed in Lemma \ref{lem:one_nghb}, there is a lower bound on their spectral gap uniformly on $M$.
 
We now prove the existence of a multiplicative decomposition.
\begin{lemma}[Local frequency decomposition]\label{lem:decomp}
 Fix $v\in \mathcal M_k$ and set $g=g_v$. There is an $\varepsilon_0>0$ such that for every $u\in \mathcal C_k([g_0])$ with $\|u v^{-1}-1\|_{W^{1,2k}(g)}<\epsilon_0$ there are unique  $c>0$, $\xi\in K_g$, $z\in E^{\me}_{\Lambda,g}$, and $y\in E^{\hi}_{\Lambda,g}\cap W^{1,2k}(g)$  close to $(1,0,0,0)$ such that \begin{equation}
  \label{eq:decomp}u=cv(\alpha_g(\xi)+z+y)\,.
   \end{equation} Moreover, for $p\in\{2,2k\}$, we find a constant $C>0$ depending on $\Lambda$, $M$, $g$, $k$, and the spectral gap of the Hessian $Q_g$ such that \begin{equation}\label{eq:mult_bdd}
      C^{-1} \|uv^{-1}  -1\|_{W^{1,p}(g)}\leq  |c-1|+|\xi|+\|z\|_{W^{1,p}(g)}+\|y\|_{W^{1,p}(g)}\leq C \|uv^{-1}  -1\|_{W^{1,p}(g)}\,. \end{equation} 
       \end{lemma}

\begin{proof}
To construct $c$, $\xi$, $y$, and $z$, we apply the inverse function theorem to the decomposition in \eqref{eq:decomp} as a function $ \Xi:\mathcal U_v\to W^{1,p}(g)$ with \begin{equation*}
  \Xi(c,\xi,z,y)\coloneqq c v(\alpha_g(\xi)+z+y)
\end{equation*} on a sufficiently small neighborhood $\mathcal U_v$ of $(1,0,0,0)$ in $\R\times K_g\times E_{\Lambda,g}^\me\times (E_{\Lambda,g}^\me\cap W^{1,p}(g))$. Since $\alpha_g(0)=1$ and $D_0\Phi_g=0$ by Lemma \ref{lem:ls}, its differential is given by $$D_{(1,0,0,0)}\Xi(c,\xi,z,y)=v(c+\xi+z+y)\,,$$ and thus an isomorphism by \eqref{eq:space_decomp}. To this end, we have to verify boundedness of the projections onto the spaces in \eqref{eq:space_decomp} in $W^{1,p}(g)$. For the low frequency spaces, boundedness of the finite-rank projections in $W^{1,p}(g)$ is immediate, and the projection onto $E_{\Lambda,g}^\hi$ can be expressed in terms of the identity and the lower frequency projections. Multiplication by $v$ is also a bounded isomorphism on $W^{1,p}(g)$ as $v$ is smooth, positive, and bounded away from $0$. Uniqueness of the inverse function theorem shows that the decompositions in $W^{1,2}(g)$ and $W^{1,2k}(g)$ coincide on a potentially smaller neighborhood. 

We know that the derivatives of $\Xi$ and by the inverse function theorem the derivatives of $\Xi^{-1}$ are locally bounded. After shrinking the domains of $\Xi$ and $\Xi^{-1}$ once more if necessary, we can ensure that the domains are convex, so line segments stay in the domains. Then by the mean value theorem, we find that
 \begin{align*}
     \Xi(c,\xi,z,y)-\Xi(1,0,0,0)&=\int_0^1 D_{1+t(c-1),t\xi,tz,ty}(\Xi)[c-1,\xi,z,y]\,\mathrm dt \,,\\\Xi^{-1}(u)-\Xi^{-1}(v)&=\int_0^1 D_{v+t(u-v)}(\Xi^{-1})[u-v]\,\mathrm d t\,,
 \end{align*} so taking the supremum norm of the differential yields the bounds in \eqref{eq:mult_bdd}. Multiplication by  $v$ and $v^{-1}$ is bounded on $W^{1,p}(g)$, so $\|u-v\|_{W^{1,p}(g)}$ is equivalent to $\|uv^{-1}-1\|_{W^{1,p}(g)}$. This proves the last part \eqref{eq:mult_bdd}.\end{proof}

 \subsection{Deficit splitting and frequency bounds}
For $u\in \mathcal C_k([g_0])$ decomposed as in Lemma~\ref{lem:decomp}, we write $$u=cv(\alpha_g(\xi)+z+y)=cv(\alpha_g(\xi)+z)(1+\eta)$$ with $\eta\coloneqq y/(\alpha_g(\xi)+z)$ and consider the metrics 
$$\tilde g\coloneqq (\alpha_g(\xi)+z)^{\frac{4k}{n-2k}} g\qquad \text{and} \qquad g_u=u^{\frac{4k}{n-2k}}g_0= (c(1+\eta))^{\frac{4k}{n-2k}} \tilde g\,.$$

For sufficiently small $|\xi|+\|z\|_{W^{1,2}(g)}$, the factors $(\alpha_g(\xi)+z)$ and $(1+\eta)$ inherit the smoothness and positivity of $u$, and the intermediate metric $\tilde g\coloneqq (\alpha_g(\xi)+z)^{4k/(n-2k)} g$ is $k$-admissible. The metric identity above and $k$-admissibility of $g_u$ also give $1+\eta\in \mathcal C_k([\tilde g])$. These properties are an advantage over the additive decomposition in \cite{FrankPeteranderl2024_sigma2} and are used below. By conformal covariance and scaling invariance, $F_{k,g_0}[u]=F_{k,g}[\alpha_g(\xi)+z+y]$, so we obtain the exact splitting 
\begin{align}\notag
    F_{k,g}[u]-Y_k=&\ F_{k,g}[\alpha_g(\xi)]-Y_k+ (F_{k,g}[\alpha_g(\xi)+z] -F_{k,g}[\alpha_g(\xi)]) \notag\\\notag&+ \left(F_{k,\tilde g}[1+\eta]-F_{k,\tilde g}[1]-D_1F_{k,\tilde g}[\eta]\right)+D_{\alpha_g(\xi)+z}F_{k, g}[y]\\\eqqcolon&\ I^\lo(\xi)+ I^{\me}(\xi,z)+I^\hi(\eta)+\tilde I(\xi,z,y)\,.\label{eq:deficit_split}\end{align}
 Here $D_1F_{k,\tilde g}[\eta]=D_{\alpha_g(\xi)+z}F_{k, g}[y]$ follows immediately from $F_{k,\tilde g}[1+t\eta]=F_{k, g}[\alpha_g(\xi)+z+ty]$. We dealt with the low frequencies using the {\L}ojasiewicz inequality in Subsection \ref{subsec:LSL}. In the remainder of this subsection, we deal with the medium frequency term $I^{\me}$ and the first order derivative term $\tilde I$. We then treat the more involved higher frequency term $I^\hi$. After fixing finitely many neighborhoods as in Lemma \ref{lem:one_nghb}, the constant $C>0$ can be chosen uniformly over these neighborhoods by taking the minimum.

\begin{lemma}[Medium frequency term]\label{lem:med} If $|\xi|+\|z\|_{W^{1,2k}(g)}$ is sufficiently small, then there is a constant $C>0$ such that $$I^\me(\xi z)\geq C \|z\|_{W^{1,2}(g)}^2\,.$$
\end{lemma}

\begin{proof}
    The first order equation \eqref{eq:LS_red_der} from the Lyapunov--Schmidt reduction gives $D_{\alpha_g(\xi)}F_{k,g}[z]=0$. As the Hessian has a positive spectral gap on $E_{\Lambda,g}^\me$ at $\xi=0$, continuity in $\xi$ and $z$ and Taylor's theorem prove the claim. Note that all norms on $E_{\Lambda,g}^\me$ are equivalent.
\end{proof}

\begin{lemma}[Mixed frequency term] \label{lem:mix}
    If $|\xi|+\|z\|_{W^{1,2k}(g)}+\|y\|_{W^{1,2k}(g)}$ is sufficiently small, then there is a $C>0$ such that $$|\tilde I(\xi,z,y)|\leq C(|\xi|+\|z\|_{W^{1,2}(g)})\|z\|_{W^{1,2}(g)}\|y\|_{W^{1,2}(g)}\,.$$
\end{lemma}

Therefore, this term is of higher order  and hence negligible for our local analysis. 

\begin{proof} The equation \eqref{eq:LS_red_der} from the Lyapunov--Schmidt reduction gives $D_{\alpha_g(\xi)}F_{k,g}[y]=0$. Hence, by the fundamental theorem of calculus, we obtain $$D_{\alpha_g(\xi)+z}F_{k,g}[y]=\int_0^1 D^2_{\alpha_g(\xi)+tz} F_{k,g}[z,y]\,\mathrm dt\,,$$ which vanishes for $z=0$. Recall that all norms of $z$ in the fixed finite-dimensional space $E^\me_{\Lambda,g}$ are equivalent. Note that $D_1^2F_{k,g}[z,y]=0$ by orthogonality. Since this integral density depends smoothly on the finite-dimensional variables $\xi$ and $z$, Taylor's theorem together with the Cauchy--Schwarz inequality provides the desired estimate. 
\end{proof}

Although $W^{1,2k}$ does not embed into $L^\infty$ for $n>2k$, positive scalar curvature gives the lower bound needed to pass from $V=(1+\eta)^{-q/n}$ to $\eta$. This will be the final ingredient for the proof of our high frequency bound.

\begin{lemma}[A Harnack lower bound]\label{lem:Har} Let $h$ be a smooth metric on a closed, connected manifold~$M$. There are $\varepsilon>0$ and $\alpha_*>0$  depending on $h$, $n$, and $k$ such that for all $1+y\in \mathcal C_k([h])$ with $\|y\|_{W^{1,2k}(h)}<\varepsilon$ we have $$1+y\geq  \alpha_*\,.$$
\end{lemma}

If we fix the spectral cutoff $\Lambda$, taking $\xi$ and $z$ sufficiently small ensures that the background metrics $\tilde g$ belong to a sufficiently small $C^2$-neighborhood of $g$ that can be chosen independently of $\Lambda$. The proof of the lemma gives constants $\varepsilon$ and $\alpha_*$ that are uniform in $\tilde g$ and independent of $\Lambda$.

\begin{proof}
    We use the Yamabe parametrization, that is, $$W=(1+y)^{\frac{k(n-2)}{n-2k}}\,.$$ Since $\Gamma_k^+\subset \Gamma_1^+$, $W^{4/(n-2)}h$ has positive scalar curvature. Thus, the conformal transformation formula for the scalar curvature  gives $$-\frac{4(n-1)}{n-2}\Delta_{h} W+R_{h} W=R_{(1+y)^{\frac{2q}n}h} W^{\frac{n+2}{n-2}}>0\,;$$ see \cite{LeeParker1987}, for instance. Thus, $W$ is a positive supersolution of a uniformly elliptic linear second-order equation with uniformly bounded zeroth order coefficient. By Harnack's inequality (see \cite[Theorem 8.18]{GilbargTrudinger2001}, for instance) applied on a finite cover of balls under coordinate charts, we find \begin{equation}\label{eq:Harnack}\|W\|_{L^{r}(h)}\leq C\inf_M W
    \end{equation} for some $C(h,n,r)>0$ and $1\leq r< n/(n-2)$. Here we used connectedness of $M$. The constants are uniform for metrics in a sufficiently small $C^2$-neighborhood of $h$.
    
   Then Chebyshev's inequality yields
    $\vol_{h}(\{1+y<1/2\})\leq 2^q\|y\|^q_{L^q(h)}$, which is small by assumption and Sobolev's embedding. Choosing $\varepsilon$ small enough such that $ 2^q\|y\|^q_{L^q(h)}\leq \vol_{h}(M)/2$, we know that \begin{equation}
        \label{eq:vol_low_bdd}
        \vol_{h}\left(\left\{1+y\geq \frac 12\right\}\right)\geq\frac 12 \vol_{h}(M)\,,
    \end{equation} which can be bounded from below uniformly in $h$ in that fixed neighborhood. Combining \eqref{eq:Harnack} and \eqref{eq:vol_low_bdd} implies that $1+y$ is uniformly bounded from below.
\end{proof}

Now we are in position to prove our remaining local frequency bound.

\begin{lemma}[High frequency term]\label{lem:hig}
    There is a $\Lambda_0>0$ such that for all $\Lambda>\Lambda_0$ outside the spectrum of the Hessian, there are $\rho_\Lambda, c_{\Lambda,2},c_{\Lambda,2k}>0$ with the following property: If $u\in \mathcal C([g])$ with $$|\xi|+\|z\|_{W^{1,2k}(g)}+\|y\|_{W^{1,2k}(g)}\leq \rho_\Lambda\,,$$ then
\begin{equation}\label{eq:I_hi_bdd}
     I^\hi(\eta)\geq c_{\Lambda,2}\|y\|_{W^{1,2}(g)}^2+c_{\Lambda,2k}\|y\|_{W^{1,2k}(g)}^{2k}\,.
  \end{equation}
\end{lemma}
Note that the smallness assumptions on $\xi$ and $z$ are needed, so we can apply Harnack's inequality, have admissibility in terms of the intermediate metric $\tilde g$, and comparability of $\eta$ and~$y$.

\begin{proof}
   The goal is to apply Proposition \ref{prop:glob_W12k}. 
   Choose a sufficiently small neighborhood of $g$ in $C^m$, independently of $\Lambda,$ and $m>0$ sufficiently large such that the constants below are uniformly controlled. After fixing $\Lambda_0$, we choose $\rho_\Lambda$ sufficiently small such that $\tilde g$ lies in this neighborhood and Lemma~\ref{lem:Har} gives $1+\eta\geq \alpha_*$.

   If we denote $$V\coloneqq (1+\eta)^{-\frac qn}\qquad \text{and} \qquad \tilde E_{k,\tilde g}[w]\coloneqq \int_M\sigma_k\left( w^{\frac {2q}n}\tilde g\right) \,\mathrm d \vol_{w^{\frac {2q}n}\tilde g}$$ for $w\in \mathcal C_k([\tilde g])\cap W^{1,2k}(\tilde g)$, then $\tilde E_{k,\tilde g}[1+\eta]=E_{k,\tilde g}[V]$. The Harnack bound $1+\eta\geq \alpha_*$ and Taylor's theorem imply that there is some $C>0$ such that $|V-1+q\eta/n|\leq C\eta^2$. Hence, the first variation formulas give \begin{equation*}
    |D_1  E_{k,\tilde g}[V-1]- D_1\tilde E_{k,\tilde g}[\eta]|=(n-2k)
    \left|\int_M \left(V-1+\frac{2q}n \eta\right)\sigma_k(\tilde g) \,\mathrm d \vol_{\tilde g}\right|\leq  C\|\eta\|_{L^2(\tilde g)}^2
\end{equation*} for some $C>0$. Applying Proposition \ref{prop:glob_W12k} thus yields 
\begin{equation}
    \tilde E_{k,\tilde g}[1+\eta]-  \tilde E_{k,\tilde g}[1]-D_1  \tilde E_{k,\tilde g}[\eta]\geq c_{n,k}(C_{\tilde g}\|\nabla_{\tilde g} \eta\|_{L^2(\tilde g)}^2+\|\nabla_{\tilde g} \eta\|_{L^{2k}(\tilde g)}^{2k})-C\|\eta\|_{L^2(\tilde g)}^2\,.\label{eq:prp_appl}
\end{equation} 
 Note further that the constants can be chosen uniformly with respect to $\tilde g$ in the fixed neighborhood. 
 
 To derive a lower bound for $I^\hi$, we have to take the volume into account. Denote $$\vartheta\coloneqq \left(\frac{1}{\vol_{\tilde g}(M)} \int_M (1+\eta)^q \,\mathrm d \vol_{\tilde g}\right)^{-\frac{2k}q}$$ Note that for $\|\eta\|_{W^{1,2k}(g)}$ sufficiently small we have $1/2\leq \vartheta\leq 3/2$. By $|(1+s)^q-1-qs|\leq C (s^2+|s|^q)$, $s>-1$, for some other $C>0$ and Taylor's theorem applied to $r\mapsto (1+r)^{-2k/q}$ at $r=0$, we can bound 
\begin{equation}\label{eq:help2}
    \left|\vartheta-1+\frac{2k}{\vol_{\tilde g}(M)} \int_M \eta \,\mathrm d \vol_{\tilde g}
    \right|\leq C(\|\eta\|_{L^2(\tilde g)}^2+\|\eta\|_{L^q(\tilde g)}^q)\,.
\end{equation}
Since we obtain by direct computation $$\vol_{\tilde g} (M)^{\frac{n-2k}n} D_1F_{k,\tilde g}[\eta]=D_1 \tilde E_{k,\tilde g}[\eta]  - \frac{2k\tilde E_{k,\tilde g}[1]}{\vol_{\tilde g}(M)}\int_{M} \eta \,\mathrm d\vol_{\tilde g}\,,$$ we obtain
\begin{align}
    \notag \vol_{\tilde g}(M)^{\frac{n-2k}n} I^{\hi}(\eta)=& \ \vartheta (\tilde E_{k,\tilde g}[1+\eta]-\tilde E_{k,\tilde g}[1]-D_1\tilde E_{k,\tilde g}[\eta])+(\vartheta-1)D_1\tilde E_{k,\tilde g}[\eta]\\&+\tilde E_{k,\tilde g}[1] \left(\vartheta-1+\frac{2k}{\vol_{\tilde g}(M)} \int_M \eta \,\mathrm d \vol_{\tilde g}\right).\label{eq:exp_eta}
\end{align} Note that $|D_1\tilde E_{k,\tilde g}[\eta]|\leq C\|\eta\|_{L^2(g)}$ for some $C>0$. Hence, combining \eqref{eq:prp_appl} with \eqref{eq:exp_eta} and \eqref{eq:help2} yields \begin{equation*}I^\hi(\eta)\geq c\left(\|\nabla_{\tilde g} \eta\|_{L^2(\tilde g)}^2+ \|\nabla_{\tilde g} \eta\|_{L^{2k}(\tilde g)}^{2k}\right)-C\left(\|\eta\|_{L^2(\tilde g)}^2+ \| \eta\|_{L^q(\tilde g)}^q\right)
    \end{equation*} for some constants $c,C>0$. 

    To return to $y=\eta(\alpha_g(\xi)+z)$, we note that $\alpha_g(\xi)+z$ is close to $1$ in any norm, so also in the $C^1$-sense. Thus, we can find $\tilde c, \tilde C>0$ such that 
    $$\|\nabla_{\tilde g} \eta\|_{L^p(\tilde g)}^p\geq \tilde c\|\nabla y\|^p_{L^p(g)}-\tilde C\|y\|^p_{L^p(g)}\qquad \text{and}\qquad \|\eta \|_{L^r(\tilde g)}^r\leq \tilde C \|y \|_{L^r(g)}^r$$ for $p\in \{2,2k\}$ and $r\in \{2,q\}$. As $2\leq2k\leq q$, we can interpolate $\|y\|^p_{L^p(g)}$ and find
\begin{equation}\label{eq:I_hi_ybound}I^\hi(\eta)\geq c\left(\|\nabla y\|_{L^2(g)}^2+ \|\nabla y\|_{L^{2k}( g)}^{2k}\right)-C\left(\|y\|_{L^2( g)}^2+ \| y\|_{L^q( g)}^q\right)
    \end{equation} with some new constants $c,C>0$, depending on the fixed neighborhood but not on $\Lambda$.

    If we assume $\Lambda$ sufficiently large, the spectral gap gives $$\|y\|_{L^2( g)}^2\leq \frac{C}\Lambda \|\nabla  y\|_{L^2( g)}^2$$ for a $C=C(g,k)>0$ independent of $\Lambda$. Indeed, $Q_g(y)\geq \Lambda \|y\|_{L^2(g)}^2$ and $Q_g(y)\leq \Lambda \|\nabla y\|_{L^2(g)}^2$ for $y$ with mean zero. By Sobolev embedding, the assumption $\|y\|_{W^{1,2k}(\tilde g)}\leq \rho_\Lambda$, and Poincar\'e's inequality, we can bound 
    $$\|y\|_{L^q( g)}^q\leq C \|y\|_{W^{1,2k}( g)}^q\leq C \rho_\Lambda^{q-2k}\|y\|_{W^{1,2k}( g)}^{2k}\leq C \rho_\Lambda^{q-2k}\|\nabla y\|_{L^{2k}( g)}^{2k}\,,$$ where the $C$ changed in the last step. Thus, after choosing $\Lambda$ sufficiently large and $\rho_\Lambda>0$ sufficiently small, we can absorb both negative terms in \eqref{eq:I_hi_ybound}. This proves the desired high-frequency estimate \eqref{eq:I_hi_bdd}.
\end{proof}

\subsection{Proof of Proposition \ref{prop:loc}}

At each minimizer, we choose a neighborhood on which the decomposition from Lemma \ref{lem:decomp} and the previous frequency estimates from \eqref{eq:low} and Lemma~\ref{lem:med}, \ref{lem:mix}, and \ref{lem:hig}  hold. Applying Lemma \ref{lem:one_nghb} to these neighborhoods provides a finite cover up to conformal diffeomorphisms. Note that in case $M$ is conformally equivalent to the sphere there is only one minimizer up to M\"obius transformation, so one fixed neighborhood at the normalized round minimizer suffices. We denote by $2+\gamma$ the maximum of the corresponding exponents from the {\L}ojasiewicz inequality~\eqref{eq:low}.

Choose $\lambda_0>0$, $\Psi_0\in G$, and $v_0\in \mathcal M_k$ such that $$\|\lambda_0 v_0^{-1}(u)_{\Psi_0}-1\|^{2+\gamma}_{W^{1,2}(g_{v_0})}+\|\lambda_0 {v_0}^{-1}(u)_{\Psi_0}-1\|^{\max\{2k,2+\gamma\}}_{W^{1,2k}(g_{v_0})}\leq 2\dist_{k,\gamma}(u)\,.$$ After applying a conformal transformation to enter one of these neighborhoods and absorbing it into $\Psi_0$, Lemma \ref{lem:decomp} gives 
$$\lambda_0(u)_{\Psi_0}=cv(\alpha_g(\xi)+z+y)$$ with $g=g_v$ for some $(c,\xi, z,y)$ close to $(1,0,0,0)$ if $\rho_{\loc}$ is small; cf.~the bounds in \eqref{eq:mult_bdd}. Here $v$ is the center of one of the neighborhoods.

Using the notation from Subsection \ref{subsec:LSL}, we  define $$\mathcal N_g\coloneqq \{\zeta\in B_{\tilde \delta}(0): \Sigma_g(\zeta)=Y_k\} \qquad \text{and}\qquad \dist(\xi,\mathcal N_g)\coloneqq\inf_{\zeta\in \mathcal N_g}|\xi-\zeta |\,,$$  where $B_{\tilde \delta}(0)\subset \mathcal U_g$ and $|\xi|< \tilde \delta/2$. Here $\mathcal N_g$ denotes the zero set $\mathcal Z$ from
\eqref{eq:low} for the choice $\Sigma=\Sigma_g-Y_k$. By conformal covariance and scaling invariance, we find $F_{k,g_0}[u]=F_{k,g}[\alpha_g(\xi)+z+y]$. Inserting the frequency bounds from \eqref{eq:low} and Lemma \ref{lem:med}, \ref{lem:mix}, and \ref{lem:hig} into \eqref{eq:deficit_split}, applied to $\lambda_0 (u)_{\Psi_0}$, gives
\begin{equation}
    F_{k,g_0}[u]-Y_k\geq c(\dist(\xi,\mathcal N_g)^{2+\gamma}+\|z\|_{W^{1,2}(g)}^2+\|y\|_{W^{1,2}(g)}^2+\|y\|_{W^{1,2k}(g)}^{2k})\,. \label{eq:combine_bdd}
\end{equation}
Here we absorbed the mixed term from the estimate of $\tilde I$ in Lemma \ref{lem:mix}. 

Take a $\zeta\in \mathcal N_g$ with $|\xi-\zeta|=\dist(\xi,\mathcal N_g)$. We restrict $\xi$ to a sufficiently small neighborhood of $0$ to ensure that this distance is attained in $\mathcal N_g\cap \mathcal U_g$. To obtain an admissible competitor for $\dist_{k,\gamma}$, we normalize $v\alpha_g(\zeta)$ in $L^q(g_0)$, which gives a minimizer $\tilde v\in \mathcal M_k$. After absorbing this normalization and $c$ into $\lambda>0$, we obtain $$\lambda \frac{(u)_{\Psi_0}}{\tilde v} -1=\frac{\alpha_g(\xi)-\alpha_g(\zeta)+z+y}{\alpha_g(\zeta)}\,.$$ 
By smoothness and positivity of $\alpha_g$ and equivalence of norms, we can bound $|\alpha_g(\xi)-\alpha_g(\zeta)|$ by $\dist(\xi,\mathcal N_g)$ up to a multiplicative constant. Therefore, we deduce
$$\left\|\lambda \frac{(u)_{\Psi_0}}{\tilde v} -1\right\|_{W^{1,p}(g_{\tilde v})}\leq C(\dist(\xi,\mathcal N_g)+\|z\|_{W^{1,p}(g)}+\|y\|_{W^{1,p}(g)})$$ for $p\in \{2,2k\}$ and some $C>0$. By our smallness assumption, inserting these parameters as competitors in the definition of $\dist_{k,\gamma}$ implies 
$$\dist_{k,\gamma}(u)\leq C(\dist(\xi,\mathcal N_g)^{2+\gamma}+\|z\|_{W^{1,2}(g)}^2+\|y\|_{W^{1,2}(g)}^2+\|y\|_{W^{1,2k}(g)}^{2k})\,.$$ Together with \eqref{eq:combine_bdd}, this proves the claim. Note that all constants can be chosen uniformly over finitely many neighborhoods as described before.

In the non-degenerate case, we have seen in Subsection \ref{subsec:LSL} that $\gamma=0$ can be taken in every neighborhood, which finishes the proof.

\section{From local to global stability}
\label{sec:4}
In this section, we prove Theorem \ref{thm:1} and Corollary \ref{cor:glob2loc}. Thus, instead of treating conformal metrics that are already close to being minimizing, we here consider arbitrary smooth, $k$-admissible conformal metrics. This uses the local estimate from Section \ref{sec:3} together with a flow argument.

\subsection{Compactness of almost minimizers via the Guan--Wang flow}\label{subsec:GW}

Let $(g_t)_{t\in [0,\infty)}$ be a family of smooth, $k$-admissible metrics of unit volume. Consider the logarithmic flow $$\frac d{dt} g_t=-(\log(\sigma_k(g_t))-\log(r_k(g_t)))g_t\,\quad \text{with}\quad  \log(r_k(g_t))\coloneqq \frac{1}{\vol_{g_t}(M)}\int_M\log(\sigma_k(g_t))\,\mathrm d \vol_{g_t}.$$ 
Guan--Wang \cite[Theorem 1]{GuanWangFlow2003} proved several properties of this flow, among others the ones that we will use in the following: First, the flow exists globally. Secondly, the normalized energy decreases, while the solutions remain admissible and volume-normalized along the flow. Lastly, along a subsequence of times $(t_j)$, the solution converges in $C^{4,\alpha}$, $0<\alpha<1$, to a smooth, $k$-admissible constant $\sigma_k$-curvature solution while $\sigma_k(g_{t_j})$ converges in $L^2$ to its positive constant curvature. The $C^{4,\alpha}$-convergence can be found in \cite[Theorem 2]{GW04}. Note that for $l=0$ their proof does not require orientability. Therefore, we have $\mathcal F_{k}[g_{t_j}]\to \mathcal F_{k}[g_\infty]$ as $j\to\infty$ and in particular $W^{1,2k}$-convergence of the conformal factors.

We rely on a constructive flow argument, which was developed by Dolbeault--Esteban--Figalli--Frank--Loss \cite{Dolbeault2025} and uses an idea by Christ \cite{christ_2017}; see also \cite{chen_2024,chen_2025,chen_2025b, ChenLuTangWang2026}. However, our argument starts from the non-constructive compactness of almost-minimizing critical points, so establishing an explicit stability constant is beyond the scope of this paper; cf.~Proposition \ref{prop:cpct}.

The first variation formula \eqref{eq:E'} yields
\begin{align*}\frac{d}{dt} \mathcal F_k[g_t]&=-\frac{n-2k}2 \int_M (\log(\sigma_k(g_t))-\log(r_k(g_t)))\sigma_k(g_t)\,\mathrm d\vol_{g_t}\\&=-\frac{n-2k}2 \int_M (\log(\sigma_k(g_t))-\log(r_k(g_t)))(\sigma_k(g_t)-r_k(g_t))\,\mathrm d\vol_{g_t}\leq 0\,,
\end{align*} where we used the definition of $r_k(g_t)$. If the flow starts from unit-volume $g_{u_0}$, then volume preservation gives $$g_t=u_t^{\frac{2q}n}g_0\qquad \text{and}\qquad \int_M u_t^{q} \,\mathrm d\vol_{g_0}=1\,.$$
On every compact time interval, parabolic regularity makes $t\mapsto u_t$ a $C^1$-curve, so in particular $u_t\in W^{1,2k}(g_0)$.

After introducing the Guan--Wang flow, we promote the compactness for solutions with almost minimal energy to compactness of admissible functions with almost minimal energy in the following proposition.

\begin{proposition}\label{prop:flow}
    Let $2\leq k<n/2$ and $(M^n,[g])$ be a smooth, closed, connected, and locally conformally flat Riemannian manifold with $k$-admissible metric $g$. For the exponent $\gamma$ from Proposition \ref{prop:loc}, there are $\varepsilon_*,c_*>0$, depending only on $M,[g],k$, such that every normalized $u\in \mathcal C_k([g])$ with $F_{k,g}[u]-Y_k\leq \varepsilon_*$ satisfies \begin{equation}
        \label{eq:quant_stab0}F_{k,g}[u]-Y_k\geq c_* \dist_{k,\gamma}(u)\,.
    \end{equation}
\end{proposition}

\begin{proof}
Let $\delta^*>0$ be chosen as in Lemma \ref{lem:att} such that the map $\min\{\dist_{k,\gamma},\delta^*\}$ is continuous on normalized functions in $W^{1,2k}(g)$. Let $\rho_{\loc},c_{\loc}>0$ be chosen as in Proposition \ref{prop:loc} such that $F_{k,g}[u]-Y_k\geq c_{\loc}\dist_{k,\gamma}(u)$ if $\dist_{k,\gamma}(u)<\rho_{\loc}$.
    
    Choose $$0<\theta<\min\{\rho_{\loc},\delta^*/2\}\,.$$ Further choose $\varepsilon_*>0$ sufficiently small such that every normalized, constant $\sigma_k$-curvature metric $g_h$ with $F_{k,g}[h]-Y_k\leq \varepsilon_*$ satisfies $\dist_{k,\gamma}(h)<\theta/2$ by Proposition \ref{prop:cpct}. 
    
  If 
    $\dist_{k,\gamma}(u)\leq \theta$, then local stability, Proposition \ref{prop:loc}, proves \eqref{eq:quant_stab0}. Assume from now on 
    $\dist_{k,\gamma}(u)> \theta$. We run the Guan--Wang flow $u_t$ from $u_0=u$. As described at the beginning of this subsection, there is a sequence $(t_j)$ such that $u_{t_j}$ converges to a critical point $u_\infty\in \mathcal C_k([g])$. By monotonicity of the flow and volume preservation, we find $$0\leq F_{k,g}[u_\infty]-Y_k\leq F_{k,g}[u_0]-Y_k\leq \varepsilon_*\,.$$
Thus, we have $\dist_{k,\gamma}(u_\infty)<\theta/2$, and by $W^{1,2k}$-convergence we find some $T>0$ with $$\dist_{k,\gamma}(u_T)<\theta\,.$$ 
   
   As the curve $\{u_t\}_{t\leq T}$ is continuous and its image is compact in $W^{1,2k}(g)$, Lemma \ref{lem:att} implies that 
   $\min\{\dist_{k,\gamma}(u_t),\delta^*\}$ is continuous in $t$. Hence, by the intermediate value theorem, there is $\tau\in (0,T)$ such that $\dist_{k,\gamma}(u_\tau)=\theta$. Yet another time by our choice of $\theta$ and monotonicity, we see that $u_0=u$ satisfies $$F_{k,g}[u_0]-Y_k\geq F_{k,g}[u_\tau]-Y_k\geq c_{\loc} \theta\,.$$ Since $\dist_{k,\gamma}(u)\leq 2$ by \eqref{eq:dist_bound}, choosing $c_*=c_\loc \theta /2$ gives \eqref{eq:quant_stab0}. 
    
   Together with the case $\dist_{k,\gamma}(u)\leq \theta$, this proves \eqref{eq:quant_stab0} with $c_*=c_\loc \min\{1,\theta /2\}$.
\end{proof}

\subsection{Proof of Corollary \ref{cor:glob2loc}}
Finally, take $(u_j)$ as given in the corollary. Then \eqref{eq:quant_stab0} implies $\dist_{k,\gamma}(u_j)\to 0$. Thus, we can choose $\lambda_j>0$, $\Psi_j\in G$, and $v_j\in \mathcal M_k$ such that $$r_j\coloneqq \lambda_jv_j^{-1} (u_j)_{\Psi_j} -1\to 0$$ in $W^{1,2k}(g_{v_j})$ as $j\to\infty$. Normalization and the Sobolev inequality imply that there is some $C>0$ such that $$|\lambda_j-1|\leq \|r_j\|_{L^q(g_{v_j})}\leq C \|r_j\|_{W^{1,2k}(g_{v_j})}\to 0\,.$$  Since $v_j^{-1} (u_j)_{\Psi_j}-1=\lambda_j^{-1} r_j+(\lambda_j^{-1}-1)$ and $\vol_{g_{v_j}}(M)=1$, we can conclude that $$\|v_j^{-1} (u_j)_{\Psi_j}-1\|_{W^{1,2k}(g_{v_j})}\to 0\,.$$ This proves the corollary.  

\subsection{Proof of Theorem \ref{thm:1}}

Assume first that $u$ is normalized. If $F_{k,g_0}[u]-Y_k\leq \varepsilon_*$, then Proposition \ref{prop:flow} applies. If $F_{k,g_0}[u]-Y_k> \varepsilon_*$, then by \eqref{eq:dist_bound}   we find
$$F_{k,g_0}[u]-Y_k\geq \varepsilon_*\geq \frac{\varepsilon_*}{2}\dist_{k,\gamma}(u)\,,$$ which proves the claim for normalized $u$ with $c=\min\{c_*,\varepsilon_*/2\}$.

If $u$ is not normalized, then set $\lambda_*\coloneqq \|u\|_{L^q(g_0)}^{-1}$ such that $\lambda_*u$ is normalized with $F_{k,g_0}[\lambda_* u]=F_{k,g_0}[u]$ by homogeneity and $\dist_{k,\gamma}(\lambda_* u)=\dist_{k,\gamma}( u)$ by taking the infimum in $\lambda$. Then we can apply the estimate as in the normalized case.

\section{Sharpness of the exponents}
\label{sec:5} In this section, we prove that the exponents $2$ and $2k$ in Theorem \ref{thm:1} cannot be decreased in the spherical case. We prove optimality for the specific choice $(M,[g_0])=(\mathbb S^n,g_*)$. The mechanisms exploited here are the same as in the case $k=2$ in \cite{FrankPeteranderl2024_sigma2}. We record how to vary the argument and verify $k$-admissibility explicitly.

\subsection{Sharpness of the quadratic $W^{1,2}$-growth}

The argument in \cite{FrankPeteranderl2024_sigma2} applies here verbatim. Indeed, for $0\neq \phi \perp (\mathcal H_0\oplus \mathcal H_1)$, let $$u_\varepsilon\coloneqq \lambda_\varepsilon(1+\varepsilon\phi)\,,$$ where $\phi$ is a smooth function, $\varepsilon>0$ is small, and $\lambda_\varepsilon$ normalizes the $L^q(g_*)$-norm. Since the G\aa rding cone is open, $u_\varepsilon$ is $k$-admissible. The spectral gap of the Hessian and the same estimates for the distance show that
$$F_{k,g_*}[u_\varepsilon]-Y_k(\mathbb S^n,g_*)\qquad \text{and} \qquad \inf_{\lambda>0, \Psi\in G}\|\lambda (u_\varepsilon)_\Psi-1\|^2_{W^{1,2}(g_*)}$$ behave both asymptotically like $\varepsilon^2$.
Here we used the distance in the form of Lemma \ref{lem:dist_notions}, (ii). Hence, the quadratic exponent is optimal, in the sense that the exponent cannot be decreased close to the set of minimizers.

\subsection{Sharpness of the $2k$-th power in the $W^{1,2k}$-growth} 

We use the family of test functions from \cite{FrankPeteranderl2024_sigma2} given by $$u_{\delta,\xi}\coloneqq \lambda_{\delta,\xi} (1+\delta (1)_{\Psi_\xi})$$ with $\lambda_{\delta,\xi}$ chosen such that $u_{\delta,\xi}$ is $L^q(g_*)$-normalized. Recall that $\Psi_\xi$ was defined at the end of Subsection \ref{subsec:conf_Hess}.

We verify that $u_{\delta,\xi}$ is $k$-admissible. We consider $V^{-2}g_*$ with $V\coloneqq (1)^{-q/n}_{\Psi_\xi}$. We diagonalize the Schouten tensor of $1+\delta (1)_{\Psi_\xi}$ with respect to an orthonormal basis with one unit vector parallel to $\nabla V$ and the others orthogonal to it. Then the Schouten tensor has $n-1$ eigenvalues $\tau>0$ and one eigenvalue $\rho>-\tau (n-k)/k$. The latter follows by direct computation. Thus, for $1\leq j\leq k$ we find by definition of the $j$-th elementary symmetric polynomial that $$\sigma_j(\rho, \tau,\dots,\tau)=\binom{n-1}{j-1}\tau^{j-1}\left(\rho+\frac{n-j}{j}\tau\right)>0,$$ and hence $1+\delta (1)_{\Psi_\xi}$ is $k$-admissible.

Steps 2-4 of \cite[Subsection~5.2]{FrankPeteranderl2024_sigma2} now apply mutatis mutandis after replacing $4$ by $2k$ in the respective formulas for the volume, the total $\sigma_2$-curvature, the $W^{1,4}$-norm and the M\"obius action. We omit the details. Since the two profiles $1$ and $\delta(1)_{\Psi_\xi}$ separate asymptotically similarly to \cite{FrankPeteranderl2024_sigma2}, we can choose sequences $\delta_j\to 0$ and $|\xi_j| \to1$ such that $$F_{k,g_*}[u_{\delta_j,\xi_j}]-Y_k(\mathbb S^n,g_*)\qquad \text{and}\qquad \inf_{\lambda>0, \Psi\in G}\|\lambda (u_{\delta_j,\xi_j})_\Psi-1\|_{W^{1,2k}(g_*)}^{2k}$$ behave both asymptotically like $\delta_j^{2k}$. Hence, the exponent $2k$ is optimal.

\subsection{Sharpness of the lower exponent bound $2+\gamma$}
The last example describes an obstruction already recorded in \cite[Proposition~4.3]{EnNeSp}. 

 Recall that $\Sigma_g(\xi)=F_{k,g}[\alpha_g(\xi)]$ is the reduced functional for the Lyapunov--Schmidt graph $\alpha_g$ and that $\mathcal N_g$ is the set of functions with $\Sigma_g(\xi)= Y_k$ as defined in the proof of Proposition~\ref{prop:loc}. We consider a case where $M$ is not conformally equivalent to the sphere and $\Sigma_g-Y_k$ does not vanish identically near $0$. Let $P_m$ be its first nonzero homogeneous term $P_m$ from Taylor's theorem with even $m\geq 4$. Then there is a vector $e$ with $|e|=1$ and $P_m(e)>0$. Set $$u_t\coloneqq \frac{ v\alpha_g(te)}{\|v\alpha_g(te)\|_{L^q(g_0)}},$$ which are normalized and $k$-admissible. Then $$F_{k,g_0}[u_t]-Y_k=\Sigma_g(te)-  Y_k=  t^m P_m (e)+\mathcal O(t^{m+1})\,.$$ 
 
 For $p\in \{2,2k\}$, we denote the individual distance in the two-term distance $\dist_{k,\gamma}$ by $$d_p(u)\coloneqq \inf_{\lambda>0, \Psi\in G,w\in \mathcal M_k}\|\lambda w^{-1}(u)_\Psi-1\|_{W^{1,p}(g_w)}\,.$$
    We claim that $d_p(u_t)$ behaves asymptotically like $t$. Indeed, if we choose $w=v$ and $\Psi=\mathbbm 1$ and absorb the normalization into $\lambda$, then $$d_p(u_t)\leq \|\alpha_g(te)-1\|_{W^{1,p}(g_v)}\leq C t$$ for some $C>0$ as $\alpha_g(0)=1$. For the other bound, assume by contradiction and Lemma \ref{lem:dist_notions},~(i), that there are $t_j\to 0$, $w_j\in\mathcal M_k$, and $\lambda_j>0$ such that
$\lambda_j u_{t_j}-w_j=o(t_j)$ in $W^{1,p}(g_0)$. By compactness of $\mathcal M_k$ and normalization, we can pass to a subsequence such that $\lambda_j\to 1$ and $w_j\to v$ smoothly as $j\to\infty$. By Lemma \ref{lem:ls}, $w_j$ belongs to the graph $\alpha_g$ with parameters $\zeta_j\to 0$ and $\Sigma_g(\zeta_j)=Y_k$. Hence, for suitably chosen $\tilde \lambda_j$ we find $$\tilde \lambda_j \alpha_g(t_je)-\alpha_g(\zeta_j)=o(t_j)$$ in $W^{1,p}(g_0)$. Taking the mean and projecting onto $K_g$ yields $\tilde \lambda_j=1+o(t_j)$ and $\zeta_j=t_j e+o(t_j)$. However, this contradicts $$0=t_j^{-m}(\Sigma_g(\zeta_j)-Y_k)=P_m(t_j^{-1}\zeta_j)+\mathcal O(t_j)\to P_m(e)>0$$ as $j\to \infty$.

Thus the deficit has order $t^m$, while both distances have order $t$, which excludes every power below $m$. In particular, $\dist_{k,\gamma}(u_t)\geq d_p(u_t)^{2+\gamma}$ implies that any potential stability exponent has to be higher than $2+\gamma$.
 
    \subsection*{Acknowledgement}
    The author would like to thank Rupert Frank for his advice, helpful suggestions, and comments on a first version of this work as well as Tobias König, Julian Scheuer, and Jesse Ratzkin for some stimulating discussions on the topic of this article. We are grateful to Wei Wei for her interest and for sharing their recent preprint \cite{GWW26}. A special thanks also goes to Jeffrey Case for his encouragement to work on this project. The author would also like to thank Larry Read for suggesting the multiplicative representation in another forthcoming work, which contributed to the idea to construct a multiplicative decomposition here. Partial support was provided through the German Research Foundation grants FR 2664/3-1 and TRR 352-Project-ID 470903074 and the Studienstiftung des deutschen Volkes.


\providecommand{\etalchar}[1]{$^{#1}$}

	\end{document}